\documentclass{amsart}
\usepackage{amssymb,amsfonts,latexsym}

\newtheorem{theorem}{Theorem}[section]
\newtheorem{lemma}[theorem]{Lemma}
\newtheorem{proposition}[theorem]{Proposition}
\newtheorem{corollary}[theorem]{Corollary}
\newtheorem{claim}[theorem]{Claim}
\theoremstyle{definition}
\newtheorem{definition}[theorem]{Definition}

\newtheorem{example}[theorem]{Example}

\newtheorem{remark}[theorem]{Remark}
\newtheorem{general remarks}[theorem]{General remarks}
\newtheorem{note}[theorem]{Note}

\newcommand{\Tr}{\operatorname{Tr}}
\newcommand{\id}{\operatorname{id}}

\newcommand{\End}{\operatorname{End}}
\newcommand{\sgn}{\operatorname{sgn}}

\newcommand{\Kemer}{\operatorname{Kemer}}

\newcommand{\Ind}{\operatorname{Ind}}
\newcommand{\ord}{\operatorname{ord}}

\renewcommand{\span}{\operatorname{span}}

\newcommand{\Sym}{\operatorname{Sym}}
\newcommand{\Par}{\operatorname{Par}}

\newcommand{\ben}{\begin{enumerate}}
\newcommand{\een}{\end{enumerate}}

\begin{document}

\title[Representability and Specht problem for $G$-graded algebras]
{Representability and Specht problem for $G$-graded algebras}

\author{Eli Aljadeff}
\address{Department of Mathematics, Technion-Israel Institute of
Technology, Haifa 32000, Israel}
\email{aljadeff@tx.technion.ac.il}
\author{Alexei Kanel-Belov}
\address{Department of Mathematics, Bar-Ilan University, Ramat-Gan, Israel}
\address{Moscow institute of Open Education, Russia}
\email{kanel@mccme.ru}

\date{July 8, 2009}




\thanks {The first author was partially supported by the ISRAEL SCIENCE FOUNDATION
(grant No. 1283/08) and by the E.SCHAVER RESEARCH FUND. The second
author was partially supported by the ISRAEL SCIENCE FOUNDATION
(grant No. 1178/06) and also by the Russian Fund of Fundamental
Research, grant  $\ RFBR 08-01-91300-IND_a$.}

\begin{abstract} Let $W$ be an associative \textit{PI}-algebra over a field $F$ of
characteristic zero, graded by a finite group $G$. Let $\id_{G}(W)$
denote the $T$-ideal of $G$-graded identities of $W$. We prove: 1.
{[$G$-graded \textit{PI}-equivalence]} There exists a field
extension $K$ of $F$ and a finite dimensional
$\mathbb{Z}/2\mathbb{Z}\times G$-graded algebra $A$ over $K$ such
that $\id_{G}(W)=\id_{G}(A^{*})$ where $A^{*}$ is the Grassmann
envelope of $A$. 2. {[$G$-graded Specht problem]} The $T$-ideal
$\id_{G}(W)$ is finitely generated as a $T$-ideal. 3. {[$G$-graded
\textit{PI}-equivalence for affine algebras]} Let $W$ be a
$G$-graded affine algebra over $F$. Then there exists a field
extension $K$ of $F$ and a finite dimensional algebra $A$ over $K$
such that $\id_{G}(W)=\id_{G}(A)$.
\end{abstract}

\maketitle

\begin{section}{Introduction} \label{Introduction}

The Specht problem (see \cite{Specht}) is considered as one of the
main problems in the theory of algebras satisfying polynomial
identities. The (generalized) Specht problem asks whether for a
given class of algebras (associative, Lie, Jordan, superalgebras,
etc.), the corresponding $T$-ideals of identities are finitely
based (i.e. finitely generated as a $T$-ideal). For associative
algebras over fields of characteristic zero we refer the reader to
\cite{Kemer2}, \cite{Kemer3}, \cite{Lat}, \cite{Raz2},
\cite{Shchigolev}. For Lie algebras the reader is referred to
\cite{Drens1}, \cite{Il3}, \cite{KirKonts}, \cite{KirKontsMolev},
\cite{Vau} whereas for alternative and Jordan algebras the reader
is referred to \cite{Il1}, \cite{Il2}, \cite{Medv}, \cite{VZel}.
As for applications of ``Specht type problems" in other topics we
refer the reader to \cite{Zel2} (in pro-$p$ groups),
\cite{Kemer1}, \cite{Shes}, \cite{Zel1} (in superalgebras),
\cite{Donk}, \cite{Kemer4}, \cite{Proce2}, \cite{Proce3},
\cite{Proce4}, \cite{Raz3}, \cite{Zub} (in invariant theory and
the theory of representations) and \cite{Pio1}, \cite{Proce1} (in
noncommutative geometry). For more comprehensive expositions on
polynomial identities the reader is referred to \cite{Bah},
\cite{BR}, \cite{Drens2}, \cite{DF}, \cite{For}, \cite{Lat},
\cite{Raz1}, \cite{Raz4}.

Polynomial identities were also studied in the context of
$G$-graded algebras (again, associative, Lie, Jordan). Here we
refer the reader to \cite{AGM}, \cite{AHN}, \cite{BD}, \cite{BSZ},
\cite{BaZa}, \cite{Vas}. More generally one may consider
polynomial identities for $H$-comodule algebras (see \cite{AK},
\cite{BB}) and use them to construct versal objects which
specialize into $k$-forms (in the sense of ``Galois descent") of a
given $H$-comodule algebra over the algebraic closure of $k$.
$H$-comodule algebras may be viewed as the noncommutative
analogues of principal fibre bundles where $H$ plays the role of
the structural group (see \cite{BM1}, \cite{H}, \cite{S}).

One of the main results of this paper is a solution to the Specht
problem for $G$-graded $PI$-algebras over a field of
characteristic zero where the group $G$ is finite.

Let $W$ be an associative  \textit{PI}-algebra over a field $F$ of
characteristic zero. Assume $W = \oplus_{g\in G}W_{g}$ is $G$-graded
where $G=\{g_1=e,g_2, \ldots,g_r\}$ is a finite group. For every
$g\in G$ let $X_{g}=\{x_{1,g},x_{2,g}, \ldots\}$ be a countable set
of variables of degree $g$ and let
$\Omega_{F,G}=F\langle\{X_{g_1},\ldots,X_{g_r}\}\rangle$ be the free
$G$-graded algebra on these variables. We refer to the elements of
$\Omega_{F,G}$ as graded polynomial or $G$-graded polynomials. An
evaluation of a graded polynomial $f\in \Omega_{F,G}$ on $W$ is
\textit{admissible} if the variables $x_{i,g}$ of $f$ are
substituted (only) by elements $\widehat{x}_{i,g}\in W_{g}$. A
graded polynomial $f$ is a \textit{graded identity} of $W$ if $f$
vanishes upon any admissible evaluation on $W$. Let $\id_{G}(W)\leq
\Omega_{F,G}$ be the $T$-ideal of $G$-graded identities of $W$ (an
ideal $I$ of $\Omega_{F,G}$ is a $T$-ideal if it is closed under all
$G$-graded endomorphisms of $\Omega_{F,G}$). As in the classical case,
also here, the $T$-ideal of identities is generated by multilinear
polynomials. Moreover, we can assume the identities are
\textit{strongly homogeneous}, that is every monomial in $f$ has the
same $G$-degree (the $G$-degree of
$x_{i_1,g_{i_1}}x_{i_2,g_{i_2}}\cdots x_{i_k,g_{i_k}}$ is the
element $g_{i_1}g_{i_2}\cdots g_{i_k} \in G$). In order to state our
main results we consider first the affine case. Let $W$ be a
\textit{PI} $G$-graded affine algebra.

\begin{theorem}[$G$-graded \textit{PI}-equivalence-affine]\label{PI-equivalence-affine}
There exists a field extension $K$ of $F$ and a finite dimensional
$G$-graded algebra $A$ over $K$ such that $\id_{G}(W)=\id_{G}(A)$ (in $\Omega_{F,G}$).
\end{theorem}

\begin{theorem}[$G$-graded Specht problem-affine]\label{Specht}
The ideal $\id_{G}(W)$ is finitely generated as a $T$-ideal.
\end{theorem}

In order to state the results for arbitrary $G$-graded algebras
(i.e. not necessarily affine) recall that the {\it Grassmann
algebra} $E$ over an unspecified infinite-dimensional $K$-vector
space is a $\mathbb{Z}/2\mathbb{Z}$-graded algebra where the
components of degree zero and one, denoted by $E_{0}$ and $E_{1}$,
are spanned by products of even and odd number of vectors
respectively. For any $\mathbb{Z}/2\mathbb{Z}$-graded algebra
$B=B_{0}\oplus B_{1}$ over $K$, we let $B^{*}= B_{0}\otimes_{K}
E_{0} \oplus B_{1}\otimes_{K} E_{1}$ be the {\it Grassmann
envelope} of $B$. If the algebra $B$ has an additional
(compatible) $G$-grading, that is $B$ is $\mathbb{Z}/2\mathbb{Z}
\times G$-graded, then $B_{0}$ and $B_{1}$ are $G$-graded and we
obtain a natural $G$-grading on $B^{*}$.

Following Kemer's approach (see \cite{Kemer1}), the results above
for $G$-graded affine algebras together with a general result of
Berele and Bergen in \cite{BB} give:

\begin{theorem}[$G$-graded \textit{PI}-equivalence]\label{PI-equivalence-general}
Let $W$ be a \textit{PI} $G$-graded algebra over $F$. Then there
exists a field extension $K$ of $F$ and a finite dimensional
$\mathbb{Z}/2\mathbb{Z} \times G$-graded algebra $A$ over $K$ such
that $\id_{G}(W)=\id_{G}(A^{*})$ where $A^{*}$ is the Grassmann
envelope of $A$.

\end{theorem}

\begin{theorem}[$G$-graded Specht problem]\label{Specht-general}
The ideal $\id_{G}(W)$ is finitely generated as a $T$-ideal.
\end{theorem}

In the last section of the paper we show how to ``pass'' from the
affine case to the general case using Berele and Bergen result.

Note that the $G$-graded algebra $W$ mentioned above (affine or non-affine)
is assumed to be (ungraded) \textit{PI}, i.e. it satisfies an
ungraded polynomial identity (an algebra may be $G$-graded
\textit{PI} even if it is not \textit{PI}; for instance take $W$
a free algebra over a field $F$ generated by two or more
indeterminates with trivial $G$-grading where $G\neq \{e\}$ (i.e.
$W_{g}=0$ for $g \neq \{e\}$)). Note that Theorem
\ref{PI-equivalence-affine} and Theorem
\ref{PI-equivalence-general} are false if $W$ is not \textit{PI}.
The $G$-graded Specht problem remains open for $G$-graded
\textit{PI} but non-\textit{PI}-algebras (affine or non-affine).

\begin{remark} All algebras considered in this paper are algebras over a
fixed field $F$ of characteristic zero. Some of the algebras will
be finite dimensional over field extensions of $F$. Whenever we
say that ``there exists a finite dimensional algebra $A$ such
that..." (without specifying the field over which this occurs) we
mean that ``there exists a field extension $K$ of $F$ and a finite
dimensional algebra $A$ over $K$ such that...". Since for fields
of characteristic zero (see, e.g., \cite{For})
$\id_{G}(W)=\id_{G}(W\otimes_{F}L)$ (in $\Omega_{F,G}$) where $L$
is any field extension of $F$ it is easy to see (and well known)
that for the proofs of the main theorems of the paper we can
always assume that the field $F$ (as well as its extensions) is
algebraically closed. It is convenient to do so since over
algebraically closed fields it is easier to describe the possible
structures of $G$-graded, finite dimensional simple algebras.
\end{remark}

 In our exposition we will follow (at least partially) the
general idea of the proofs in the ungraded case as they appear in
\cite{BR}. In the final steps of the proof of Theorem
\ref{PI-equivalence-affine} we apply the Zubrilin-Razmyslov
identity, an approach which is substantially different from the
exposition in \cite{BR}. It should be mentioned however that the
authors of \cite{BR} hint that the Zubrilin-Razmyslov identity
could be used to finalize the proof. In addition to the above
mentioned difference, there are several substantial obstacles
which should be overcome when generalizing from the ungraded case
to the $G$-graded case (and especially to the case where the group
$G$ is non-abelian). Let us mention here the main steps of the
proof.

As in the ungraded case, we ``approximate'' the $T$-ideal $\id_{G}(W)$
by $\id_{G}(A)\subseteq \id_{G}(W)$ where $A$ is a $G$-graded
finite dimensional algebra. Then by induction we get ``closer'' to
$\id_{G}(W)$. The first step is therefore the statement which
``allows the induction to get started'' namely showing the
existence of a $G$-graded finite dimensional
algebra $A$ with $\id_{G}(A) \subseteq \id_{G}(W)$.
We point out that already in this step we need to assume that $G$
is finite.

In order to apply induction we represent the $T$-ideal
$\Gamma$ by a certain finite set of parameters $\{(\alpha,s)\} \subset
(\mathbb{Z}^{+})^{r}\times \mathbb{Z}^{+}$ which we call
\textit{Kemer points}. To each Kemer point $(\alpha,s)$ we attach
a certain set of polynomials, called \textit{Kemer polynomials},
which are \textit{outside} $\Gamma$. These polynomials are $G$-graded,
multilinear and have alternating sets of cardinalities as
prescribed by the point $(\alpha,s)$. This is the point where our
proof differs substantially from the ungraded case and moreover
where the noncommutativity of the group $G$ comes into play. Alternating
$G$-graded variables where $G$ is non-abelian yields monomials
which belong to different $G$-graded components. This basic fact
led us to consider alternating sets which are homogeneous, i.e. of
variables that correspond to the same $g$-component.

As mentioned above, the Kemer polynomials that correspond to the
point $(\alpha,s)$ do not belong to the $T$-ideal $\Gamma$ and
hence, if we add them to $\Gamma$ we obtain a larger $T$-ideal
$\Gamma^{'}$ of which $(\alpha,s)$ is not a Kemer point.
Consequently the Kemer points of $\Gamma^{'}$ are ``smaller''
compared to those of $\Gamma$ (with respect to a certain
ordering). By induction, there is a $G$-graded finite dimensional
algebra $A^{'}$ with $\Gamma^{'}=\id_{G}(A^{'})$.

Let us sketch the rest of proof of Theorem \ref{PI-equivalence-affine}.
Let $\mathcal{W}_{\Gamma} =
F\langle\{X_{g_1},\ldots,X_{g_r}\}\rangle/ \Gamma$ be the
relatively free algebra of the ideal $\Gamma$. Clearly
$\id_{G}(\mathcal{W}_{\Gamma})=\Gamma$. We construct a
representable algebra $B_{(\alpha,s)}$ (i.e. an algebra which
can be $G$-graded embedded in a $G$-graded matrix algebra
over a large enough field $K$) which is on one hand a
$G$-graded homomorphic image of the relatively free algebra
$\mathcal{W}_{\Gamma}$ and on the other hand the ideal $I$ in
$\mathcal{W}_{\Gamma}$, generated by the Kemer polynomials that
correspond to the point $(\alpha,s)$, is mapped isomorphically.
Then we conclude that $\Gamma = \id_{G}(B_{(\alpha,s)} \oplus
A^{'})$.

The exposition above does not reveal a fundamental feature of the
proof. One is able to prove, using Zubrilin-Razmyslov theory, that
elements of $I$ which correspond to (rather than generated by)
Kemer polynomials are mapped isomorphically into $B_{(\alpha,s)}$.
Clearly this is not sufficient. In order to show that the
``entire'' ideal $I$ in $\mathcal{W}_{\Gamma}$ is mapped
isomorphically one needs to show that any non-zero element in $I$
generates another element in $I$ which corresponds to a Kemer
polynomial. This is the so called \textit{Phoenix property}. It is
fair to say that a big part (if not the main part) of the proof of
Theorem \ref{PI-equivalence-affine} is devoted to the proof of the
Phoenix property of Kemer polynomials. This is achieved by
establishing a fundamental connection between Kemer points, Kemer
polynomials and the structure of $G$-graded finite dimensional
algebras. Here we use a key result of Bahturin, Sehgal and Zaicev
in which they fully describe the structure of $G$-graded, finite
dimensional $G$-simple algebras in terms of fine and elementary
gradings (see \cite{BSZ} and Theorem \ref{BSZ} below).

\begin{remark}
This is a second place where the noncommutativity of the group $G$
comes into play. It is not difficult to show that if $G$ is
abelian then a $G$-graded, finite dimensional simple algebra (over
an algebraically closed field $F$ of characteristic zero) is the
direct product of matrix algebras of the \underline{same degree}. This is not the
case in general (although not impossible) if $G$ is non abelian.
\end{remark}

After completing the proof of Theorem \ref{PI-equivalence-affine}
we turn to the proof of Theorem \ref{Specht} (Specht problem).
This is again based on the above mentioned result of Bahturin,
Sehgal and Zaicev. The main point is that one can deduce from
their result that if $F$ is algebraically closed then the number
of non-isomorphic $G$-gradings which can be defined on a given
semisimple algebra $A$ is finite. Interestingly, this is in
contrast to the case where the ``grading'' is given by other type
of Hopf algebras (see \cite{AEGN}, \cite{AK}). For instance if $H$
is the Sweedler algebra of dimension $4$ over the field of complex
numbers, then there exist infinitely many non-isomorphic
$H$-comodule structures on $M_{2}(\mathbb{C})$.

\begin{remark}
It is important to mention that the proof of Theorem
\ref{PI-equivalence-affine} (as, in fact, all proofs known to us
of ``Specht type problems") can be viewed as an applications of
the Grothendieck approach to noncommutative polynomials. Indeed
one has to translate properties of finite dimensional algebras $B$
(dimension of the $g$-homogeneous component of the semisimple part
of $B$ , index of nilpotency of the radical $J(B)$) which we call
``geometric" to a ``functional" or ``combinatorial" language of
polynomial identities (see Sections \ref{the index of $G$-graded
$T$-ideals}, \ref{The index of finite dimensional $G$- graded
algebra}, \ref{Section: Kemer's Lemma $1$}, \ref{Section: Kemer's
Lemma $2$}, \ref{Appendix A} and Theorems
\ref{TheoremAffineAdequateModel}, \ref{TheoremAdequateModel}).
\end{remark}

\begin{remark}

As in the ungraded case also here the solution of the Specht
problem does not yield explicit generating sets of the $T$-ideals
of identities. Nevertheless in some special cases such generating
sets were found and in particular the Specht problem was solved
(see \cite{AHN}, \cite{BD}, \cite{Vas} ).

\end{remark}

\begin{remark}[Codimension growth]

In \cite{AGM} the codimension growth of $G$-graded algebras where
$G$ is a finite \underline{abelian} group was considered. It is proved that if
$A$ is a $G$-graded finite dimensional algebra then

$$
\exp^G(A)=\lim_{n\to \infty}\sqrt[n]{c^G_n(A)}
$$

exists and is an integer. Here $c^G_n(A)$ denotes the dimension of
the subspace of multilinear elements in $n$ free generators in the
relatively free $G$-graded algebra of $A$.
\end{remark}

Applying Theorem \ref{PI-equivalence-affine} we obtain:

\begin{corollary} Let $W$ be a \textit{PI}, $G$-graded affine algebra where
$G$ is a finite abelian group. Then
$$
\exp^G(W)=\lim_{n\to \infty}\sqrt[n]{c^G_n(W)}
$$

exists and is an integer.

\end{corollary}

Before embarking into the proofs (sections $2-13$) the reader is
advised to read Appendix $A$ in which we present some of the basic
ideas which show the connection between the structure of
polynomial and finite dimensional algebras. These ideas are
fundamental and being used along the entire paper for $G$-graded
algebras. For simplicity, in this appendix, we present them for
ungraded algebras. Of course the reader who is familiar with
Kemer's proof of the Specht problem may skip the reading of this
appendix.

\begin{remark}
It came to our attention that Irina Sviridova recently obtained
similar results in case the grading group $G$ is finite
\underline{abelian}. Sviridova' results and the results of this
paper were obtained independently.
\end{remark}

\end{section}

\begin{section} {Getting started} \label{Start}

In this section we show that $\id_{G}(W)$ contains the $T$-ideal
of identities of a finite dimensional graded algebra.

\begin{proposition}

Let $G$ be a finite group and $W$ a $G$-graded affine algebra.
Assume $W$ is (ungraded) \textit{PI}. Then there exists a finite
dimensional $G$-graded algebra $A_{G}$ with $\id_{G}(A_{G})\subset
\id_{G}(W)$.

\end{proposition}

\begin{proof}

Let $I=\id(W)$ denote the ideal of ungraded identities of $W$ and
let $\mathcal{W}=F\langle X \rangle/I$ be the corresponding
relatively free algebra. We fix an epimorphism
$\eta:\mathcal{W}\longrightarrow W$. For every $g\in G$, let
$\mathcal{W}_{g}$ be the inverse image of the $g$-homogeneous
component $W_{g}$ under $\eta$. Note that $\mathcal{W}$ is not
$G$-graded (the components may intersect non-trivially), nevertheless
$\mathcal{W}_{g}\mathcal{W}_{h}\subset \mathcal{W} _{gh}$ for
every $g,h\in G$.

Let $FG$ denote the group algebra of $G$ over $F$ and let
$\mathcal{W}_{G}$ be the subalgebra of $\mathcal{W} \otimes_{F} FG$
generated by the subspaces
$\langle\{\mathcal{W}_{g}\otimes Fg\}_{g\in G}\rangle$. We
claim that

\begin{enumerate}
\item\label{a_1}

$\mathcal{W}_{G}$ is $G$-graded where
$(\mathcal{W}_{G})_{g}=\mathcal{W}_{g}\otimes Fg$, $g\in G$.

\item\label{a_2}

The map $\mathcal{W}_{G} \longrightarrow W$ given by $x\otimes g
\longmapsto \eta(x)$ is a $G$-graded epimorphism.

\end{enumerate}

Indeed, it is clear that the intersection of the spaces
$(\mathcal{W}_{G})_{g}$ and $(\mathcal{W}_{G})_{h}$ is zero for
any $g,h\in G$ with $g\neq h$. Furthermore,
$(\mathcal{W}_{G})_{g}(\mathcal{W}_{G})_{h}=(\mathcal{W}_{g}\otimes
Fg)(\mathcal{W}_{h}\otimes Fh)\subset \mathcal{W}_{gh}\otimes Fgh=
(\mathcal{W}_{G})_{gh}$ proving \ref{a_1}. The proof of \ref{a_2}
is now clear.

It follows from the claim that $\id_{G}(W)\supset
\id_{G}(\mathcal{W}_{G})$. Now, from the classical theory (see \cite{BR}, Corollary 4.9)
we know
that there exists a finite dimensional algebra $A$ such that
$\id(A)\subset \id(\mathcal{W})$. Consider the algebra $A _{G}
=A\otimes_{F}FG$, $G$-graded via $FG$.

\begin{claim}
$\id_{G}(A_{G})\subset \id_{G}(\mathcal{W}_{G})$
\end{claim}

To see this let $f(x_{i,g})\in \id_{G}(A_{G})$ be a multilinear
polynomial. As noted above we can assume that $f=f_{g}$ is $g$-
strongly homogeneous (i.e. consisting of monomials with total
degree $g\in G$). It follows from the definition that if we ignore
the $G$-degree of each variable in $f_{g}$ we obtain an identity
of $A$. But $\id(A)\subset \id(\mathcal{W})$ so $f_{g}$ is a
graded identity of $\id_{G}(\mathcal{W}_{G})$ and the claim is
proved.

Finally we have that
$\id_{G}(A_{G})\subset \id_{G}(\mathcal{W}_{G})\subset \id_{G}(W)$
and since $G$ is finite the proposition is proved.
\end{proof}

\begin{remark}

This is the point were we need to assume that the grading group
$G$ is finite. We do not know whether the statement holds for
infinite groups.

\end{remark}

\end{section}

\begin{section}{the index of $G$-graded $T$-ideals}\label{the index of $G$-graded $T$-ideals}

Let $\Gamma$ be a $G$-graded $T$-ideal. As noted in the introduction
since the field $F$ is of characteristic zero the $T$-ideal
$\Gamma$ is generated by multilinear graded polynomials which are
strongly homogeneous.

\begin{definition}

Let $f(X_{G})=f(x_{i,g})$ be a multilinear $G$-graded polynomial
which is strongly homogeneous. Let $g\in G$ and let $X_{g}$ be the
set of all $g$-variables in $X_{G}$. Let
$S_{g}=\{x_{1,g},x_{2,g},\ldots,x_{m,g}\}$ be a subset of $X_{g}$
and let $Y_{G}=X_{G}\backslash S_{g}$ be the set of the remaining
variables. We say that $f(X_{G})$ is \textit{alternating in the
set} $S_{g}$ (or that the variables of $S_{g}$ \textit{alternate
in} $f(X_{G})$) if there exists a (multilinear, strongly
homogeneous) $G$-graded polynomial
$h(S_{g};Y_{G})=h(x_{1,g},x_{2,g},\ldots,x_{m,g};Y_{G})$ such that

$f(x_{1,g},x_{2,g},\ldots,x_{m,g};Y_{G})=\sum_{\sigma\in
\Sym(m)}\sgn(\sigma)h(x_{\sigma(1),g},
x_{\sigma(2),g},\ldots,x_{\sigma(m),g};Y_{G})$.

\end{definition}

Following the notation in \cite{BR}, if $S_{g_{i_1}},
S_{g_{i_2}},\ldots,S_{g_{i_p}}$ are $p$ disjoint sets of variables of
$X_{G}$ (where $S_{g_{i_j}}\subset X_{g_{i_j}}$) we say that
$f(X_{G})$ is alternating in $S_{g_{i_1}},
S_{g_{i_2}},\ldots,S_{g_{i_p}}$ if $f(X_{G})$ is alternating in each
set $S_{g_{i_j}}$ (note that the polynomial
$x_1x_2y_1y_2-x_2x_1y_2y_1$ is not alternating in the $x's$ nor in
the $y's$).

As in the classical theory we will consider polynomials which
alternate in $\nu$ disjoint sets of the form $S_{g}$ for every
$g\in G$. If for every $g$ in $G$, the sets $S_{g}$ have the same
cardinality (say $d_{g}$) we will say that $f(X_{G})$ is
$\nu$-fold $(d_{g_{1}},d_{g_{2}},\ldots,d_{g_{r}})$-alternating.
Further we will need to consider polynomials (again, in analogy to
the classical case) which in addition to the alternating sets
mentioned above it alternates in $t$ disjoint sets $K_{g}\subset
X_{g}$ (and also disjoint to the previous sets) such that
$\ord(K_{g})=d_{g}+1$. Note that the $g$'s in $G$ that correspond
to the $K_{g}$'s need not be different.

We continue with the definition of the \textit{graded index} of a
$T$-ideal $\Gamma$. For this we need the notion of $g$-Capelli
polynomial.

Let $X_{n,g}=\{x_{1,g},\ldots,x_{n,g}\}$ be a set of $n$ variables
of degree $g$ and let $Y=\{y_{1},\ldots,y_{n}\}$ be a set of $n$
ungraded variables.

\begin{definition}

The $g$-\textit{Capelli polynomial} $c_{n,g}$ (of degree $2n$) is
the polynomial obtained by alternating the set $x_{i,g}$'s in the
monomial $x_{1,g}y_{1}x_{2,g}y_{2}\cdots x_{n,g}y_{n}$

More precisely

$c_{n,g}=\sum_{\sigma\in
\Sym(n)}\sgn(\sigma)x_{\sigma(1),g}y_{1}x_{\sigma(2),g}y_{2}\cdots
x_{\sigma(n),g}y_{n}$

\end{definition}

\begin{remark}

Note that a $g$-Capelli polynomial yields a set of $G$-graded
polynomials by specifying degrees in $G$ to the $y$'s. So when we
say that the $g$-Capelli polynomial $c_{n,g}$ is in $\Gamma$ we
mean that all the $G$-Graded polynomials obtained from  $c_{n,g}$
by substitutions of the form $y_i \longmapsto y_{i,g}$, some $g\in
G$, are in $\Gamma$.

\end{remark}

\begin{lemma}
For every $g\in G$, there exists an integer $n_{g}$ such that the
$T$-ideal $\Gamma$ contains $c_{n_{g},g}$.

\end{lemma}

\begin{proof}
Let $A_{G}$ be a finite dimensional $G$-graded algebra with
$\id_{G}(A_{G})\subset \Gamma$. Let $A_{g}$ be the $g$-homogeneous
component of $A_{G}$ and let $n_{g}=\dim(A_{g})+1$. Clearly,
$c_{n_{g},g}$ is contained in $\id_{G}(A_{G})$ and hence in
$\Gamma$.
\end{proof}

\begin{corollary}

If $f(X_{G})=f(\{x_{i,g}\})$ is a multilinear $G$-graded polynomial,
strongly homogeneous and alternating on a set $S_{g}$ of
cardinality $n_{g}$, then $f(X_{G})\in \Gamma$. Consequently there
exists an integer $M_{g}$ which bounds (from above) the
cardinality of the $g$-alternating sets in any $G$-graded
polynomial $h$ which is not in $\Gamma$.

\end{corollary}

We can now define $\Ind(\Gamma)$, the index of $\Gamma$. It will
consist of a finite set of points $(\alpha, s)$ in the lattice
$\Lambda_{\Gamma} \times D_{\Gamma} \cong
(\mathbb{Z}^{+})^{r}\times (\mathbb{Z}^{+} \cup \infty)$ where
$\alpha\in \Lambda_{\Gamma}$ and $s\in D_{\Gamma}$. We first
determine the set $\Ind(\Gamma)_{0}$, namely the projection of
$\Ind(\Gamma)$ into $\Lambda_{\Gamma}$ and then for each point
$\alpha \in \Ind(\Gamma)_{0}$ we determine $s(\alpha) \in
D_{\Gamma}$ so that $(\alpha, s(\alpha))\in \Ind(\Gamma)$. In
particular if $(\alpha, s_1)$ and $(\alpha, s_2)$ are both in
$\Ind(\Gamma)$ then $s_1=s_2$. Before defining these sets we
introduce a partial order on $(\mathbb{Z}^{+})^{r}\times
(\mathbb{Z}^{+} \cup \infty)$ starting with a partial order on
$(\mathbb{Z}^{+})^{r}$: If $\alpha = (\alpha_1,\ldots,\alpha_r)$
and $\beta= (\beta_1,\ldots,\beta_r)$ are elements in
$(\mathbb{Z}^{+})^{r}$ we put $(\alpha_1,\ldots,\alpha_r) \preceq
(\beta_1,\ldots,\beta_r)$ if and only if $\alpha_i \leq \beta_i$
for $i=1,\ldots,r$.

Next, for $(\alpha, s)$ and $(\beta, s^{'})$ in
$(\mathbb{Z}^{+})^{r}\times (\mathbb{Z}^{+} \cup \infty)$ put

$(\alpha, s) \preceq (\beta, s^{'})$ if and only if either

\begin{enumerate}

\item

$\alpha \prec \beta$ (i.e. $\alpha \preceq \beta$ and for some
$j$, $\alpha_j < \beta_j$), or

\item

$\alpha = \beta$ and $s\leq s^{'}$. (By definition $s < \infty$
for  every $s\in \mathbb{Z}^{+}$).

\end{enumerate}

\begin{definition}

A point $\alpha= (\alpha_{g_1},\ldots,\alpha_{g_r})$ is in
$\Ind(\Gamma)_{0}$ if and only if for any integer $\nu$ there
exists a multilinear $G$-graded polynomial outside $\Gamma$ with
$\nu$ alternating $g$-sets of cardinality $\alpha_{g}$ for every
$g$ in $G$.

\end{definition}

\begin{lemma}
The set $\Ind(\Gamma)_{0}$ is bounded (finite). Moreover if
$\alpha= (\alpha_{g_1},\ldots,\alpha_{g_r})\in \Ind(\Gamma)_{0}$
then any $\alpha^{'} \preceq \alpha$  is also in
$\Ind(\Gamma)_{0}$.
\end{lemma}

\begin{proof}

The first statement is a consequence of the fact that $\Gamma
\supseteq \id_{G}(A)$ where $A$ is a finite dimensional $G$-graded
algebra. The second statement follows at once from the definition
of $\Ind(\Gamma)_{0}$.
\end{proof}

The important points in $\Ind(\Gamma)_{0}$ are the extremal ones.

\begin{definition}

\begin{enumerate}

\item A point $\alpha$ in $\Ind(\Gamma)_{0}$ is \textit{extremal}
if for any point $\beta$ in  $\Ind(\Gamma)_{0}$, $\beta \succeq
\alpha \Longrightarrow \beta = \alpha$. We denote by
$E_{0}(\Gamma)$ the set of all extremal points in
$\Ind(\Gamma)_{0}$.

\item

For any point $\alpha$ in $E_{0}(\Gamma)$ and every integer $\nu$
consider the set $\Omega_{\alpha, \nu}$ of all $\nu$-folds
alternating polynomials in $g$-sets of cardinality $\alpha_{g}$
for every $g$ in $G$, which are not in $\Gamma$. For such
polynomials $f$ we consider the number $s_{\Gamma}(\alpha, \nu,
f)$ of alternating $g$-homogeneous sets (any $g\in G$) of disjoint
variables, of cardinality $\alpha_{g}+1$. Note that by the
pigeonhole principle and by the definition of $E_{0}(\Gamma)$ the
set of numbers $\{s_{\Gamma}(\alpha, \nu, f)\}_{f\in
\Omega_{\alpha, \nu}}$ is bounded. Let $s_{\Gamma}(\alpha, \nu) =
\max\{s_{\Gamma}(\alpha, \nu, f)\}_{f\in \Omega_{\alpha, \nu}}$.
Note that the sequence $s_{\Gamma}(\alpha, \nu)$ is monotonically
decreasing as a function of $\nu$ and hence there exists an
integer $\mu=\mu(\Gamma, \alpha)$ for which the sequence
stabilizes, that is for $\nu \geq \mu$ the sequence
$s_{\Gamma}(\alpha, \nu)$ is constant. We denote by $s(\alpha) =
\lim_{\mu \rightarrow \infty}(s_{\Gamma}(\alpha, \nu))=
s_{\Gamma}(\alpha, \mu)$. At this point the integer $\mu$
depends on $\alpha$. However since the set $E_{0}(\Gamma)$ is
finite we take $\mu$ to be the maximum of all $\mu$'s considered
above.

\item

We define now the set $\Ind(\Gamma)$ as the set of points
$(\alpha, s)$ in $(\mathbb{Z}^{+})^{r}\times (\mathbb{Z}^{+} \cup
\infty)$ such that $\alpha \in \Ind(\Gamma)_{0}$ and $s=
s_{\Gamma}(\alpha)$ if $\alpha \in E_{0}(\Gamma)$ (extremal) or
$s= \infty$ otherwise.

\item

Given a $G$-graded $T$-ideal $\Gamma$ which contains the graded
identities of a finite dimensional $G$-graded algebra $A$ we let
$\Kemer(\Gamma)$ (the Kemer set of $\Gamma$), to be set of points
$(\alpha, s)$ in $\Ind(\Gamma)$ where $\alpha$ is extremal. We
refer to the elements of $\Kemer(\Gamma)$ as the the Kemer points
of $\Gamma$.

\end{enumerate}

\end{definition}

\begin{remark}\label{comparison}

Let $\Gamma_{1}\supseteq \Gamma_{2}$ be $T$-ideals which contain
$\id_{G}(A)$ where $A$ is a finite dimensional $G$-graded algebra.
Then $\Ind({\Gamma_{1}})$ $\subseteq$ $\Ind({\Gamma_{2}})$. In
particular for every Kemer point $(\alpha, s)$ in
$\Ind({\Gamma_{1}})$ there is a Kemer point $(\beta, s^{'})$ in
$\Ind({\Gamma_{2}})$ with $(\alpha, s) \preceq (\beta, s^{'})$.

\end{remark}

We are now ready to define Kemer polynomials for a $G$-graded
$T$-ideal $\Gamma$.

\begin{definition}

\begin{enumerate}

\item Let $(\alpha, s)$ be a Kemer point of the $T$-ideal $\Gamma$.
A $G$-graded polynomial $f$ is called \textit{Kemer polynomial for the point} $(\alpha, s)$ if
$f$ is not in $\Gamma$ and it has at least $\mu$-folds of alternating
$g$-sets of cardinality $\alpha_{g}$ (\textit{small sets}) for
every $g$ in $G$ and $s$ homogeneous sets of disjoint variables
$Y_{g}$ (some $g$ in $G$) of cardinality $\alpha_{g}+1$
(\textit{big sets}).

\item A polynomial $f$ is {\it Kemer} for the $T$-ideal $\Gamma$
if it is Kemer for a Kemer point of $\Gamma$.

Note that a polynomial $f$ cannot be Kemer simultaneously for
different Kemer points of $\Gamma$.

\end{enumerate}
\end{definition}

\end{section}

\begin{section}{The index of finite dimensional $G$-graded
algebras}\label{The index of finite dimensional $G$- graded
algebra}

As explained in the introduction the Phoenix property of Kemer
polynomials is essential for the proof of Theorem
\ref{PI-equivalence-affine}. Since we will need this notion not
only with respect to Kemer polynomials, we give here a general
definition.

\begin{definition}(The Phoenix property)
Let $\Gamma$ be a $T$-ideal as above. Let $P$ be any property
which may be satisfied by $G$-graded polynomials (e.g. being
Kemer). We say that $P$ is ``\textit{$\Gamma$-Phoenix}'' (or in
short ``\textit{Phoenix}'') if given a polynomial $f$ having $P$
which is not in $\Gamma$ and any $f^{'}$ in $\langle f \rangle$
(the $T$-ideal generated by $f$) which is not in $\Gamma$ as well,
there exists a polynomial $f^{''}$ in $\langle f^{'} \rangle$
which is not in $\Gamma$ and satisfies $P$. We say that $P$ is
``\textit{strictly $\Gamma$-Phoenix}'' if (with the above
notation) $f^{'}$ itself satisfies $P$.

\end{definition}

Let us pause for a moment and summarize what we have at this
point. We are given a $T$-ideal $\Gamma$ (the $T$-ideal of
$G$-graded identities of a $G$-graded affine algebra $W$). We
assume that $W$ is \textit{PI} and hence as shown in Section
\ref{Start} there exists a $G$-graded finite dimensional algebra
$A$ with $\Gamma \supseteq \id_{G}(A)$. To the $T$-ideals $\Gamma$
and $\id_{G}(A)$ we attached each the corresponding set of Kemer
points. By Remark \ref{comparison} for every Kemer point
$(\alpha,s)$ of $\Gamma$ there is a Kemer point $(\beta,s^{'})$ of
$\id_{G}(A)$ (or by abuse of language, a Kemer point of $A$) such
that $(\alpha,s) \preceq (\beta,s^{'})$. Our goal is to replace
the algebra $A$ by a finite dimensional algebra $A^{'}$ with
$\Gamma \supseteq \id_{G}(A^{'})$ and such that

\begin{enumerate}

\item

The $T$-ideals $\Gamma$ and $\id_{G}(A^{'})$ have the same Kemer
points.

\item

Every Kemer polynomial of $A^{'}$ is not in $\Gamma$ (i.e.
$\Gamma$ and $A^{'}$ have the same Kemer polynomials).

\end{enumerate}

Here we establish the important connection between the
combinatorics of the Kemer polynomials of $\Gamma$ and the
structure of $G$-graded finite dimensional algebras. The
``Phoenix'' property for the Kemer polynomials of $\Gamma$ will
follow from that connection.

Let $A$ be a finite dimensional $G$-graded algebra over $F$.
It is well known (see \cite{CM}) that the Jacobson radical
$J(A)$ is $G$-graded and hence $\overline{A}=A/J(A)$ is
semisimple and $G$-graded. Moreover by the Wedderburn-Malcev
Principal Theorem for $G$-graded algebras (see \cite{StVan}) there
exists a semisimple $G$-graded subalgebra $\overline{A}$
of $A$ such that $A= \overline{A} \oplus J(A)$ as $G$-graded
vector spaces. In addition, the subalgebra $\overline{A}$
may be decomposed as a $G$-graded algebra into the
direct product of (semisimple) $G$-simple algebras
$ \overline{A}\cong A_{1}\times A_{2}\times\cdots\times A_{q}$ (by definition
$A_{i}$ is $G$-simple if it has no non-trivial $G$-graded ideals).

\begin{remark} This decomposition enables us to consider ``semisimple'' and
``radical'' substitutions. More precisely, since in order to check
whether a given $G$-graded multilinear polynomial is an identity
of $A$ it is sufficient to evaluate the variables in any spanning
set of homogeneous elements, we may take a basis consisting of
homogeneous elements in $\overline{A}$ or in $J(A)$. We refer to
such evaluations as semisimple or radical evaluations
respectively. Moreover, the semisimple substitutions may be taken
from the $G$-simple components. In what follows, whenever we
evaluate $G$-graded polynomial on $G$-graded finite dimensional algebras, we
consider only evaluations of that kind.
\end{remark}
\bigskip

In fact, in the proofs we will need a rather precise ``control" of the evaluations
in the $G$-simple components. This is
provided by a structure theorem proved by Bahturin, Sehgal and Zaicev which will play a decisive
role in the proofs of the main results of this paper.

\begin{theorem}[\cite{BSZ}]\label{BSZ}

Let $\widehat{A}$ be a $G$-simple algebra. Then there exists a
subgroup $H$ of $G$, a $2$-cocycle $f:H\times H\longrightarrow
F^{*}$ where the action of $H$ on $F$ is trivial, an integer $k$
and a $k$-tuple $(g_1=1,g_2,\ldots,g_k)\in G^{(k)}$ such that
$\widehat{A}$ is $G$-graded isomorphic to $C=F^{f}H\otimes
M_{k}(F)$ where $C_g= \span_F\{b_h \otimes E_{i,j}:
g=g_i^{-1}hg_j\}$. Here $b_h \in F^{f}H$ is a representative of
$h\in H$ and $E_{i,j}\in M_{k}(F)$ is the $(i,j)$ elementary
matrix.

In particular the idempotents $1 \otimes E_{i,j}$ as well as the identity of $\widehat{A}$ are
homogeneous of degree $e\in G$.
\end{theorem}

Let $\overline{A} = \oplus_{g\in G} \overline{A}_{g}$ be the decomposition of
$\overline{A}$ into its homogeneous components and let
$d_{A,g}=\dim_{F}(\overline{A}_{g})$, the dimension of
$\overline{A}_{g}$ over $F$, for every $g$ in $G$. Let $n_{A}$ be the
nilpotency index of $J(A)$. The following proposition establishes
an easy connection between the parameters
$d_{A,g_1},\dots,d_{A,g_r}$, $n_{A}$ and the Kemer points of $A$.
We denote by $G-\Par(A)$ the $(r+1)$-tuple
$(d_{A,g_1},\dots,d_{A,g_r},n_{A}-1)$ in
$(\mathbb{Z}^{+})^{r}\times (\mathbb{Z}^{+})$.

\begin{proposition} \label{easy direction}
If $(\alpha, s)=(\alpha_{g_1},\dots,\alpha_{g_r},s)$ is a Kemer
point of $A$ then $(\alpha, s) \preceq G-\Par(A)$.
\end{proposition}
\begin{proof}

Assume this is false. It follows that either $\alpha_{g}
> \dim_{F}(\overline{A}_{g})$ for some $g \in G$ or else $\alpha_{g} =
\dim_{F}(\overline{A}_{g})$ for every $g \in G$ and $s > n_{A}-1$.
We show that both are impossible. First recall that $(\alpha,s)$
being a Kemer point of $A$ implies the existence of polynomials
$f$ which are non-identities of $A$ with arbitrary many
alternating $g$-sets of cardinality $\alpha_{g}$. Hence, if
$\alpha_{g} > \dim_{F}(\overline{A}_{g})$ for some $g \in G$, it
follows that in each such alternating set there must be at least
one radical substitution in any nonzero evaluation of $f$. This
implies that we cannot have more than $n_{A}-1$ $g$-alternating
sets of cardinality $\alpha_{g}$ contradicting our previous
statement. Next, suppose that $\alpha_{g} =
\dim_{F}(\overline{A}_{g})$ for all $g \in G$ and $s > n_{A}-1$.
This means that we have $s$-alternating $g$-sets for some $g$'s in
$G$, of cardinality $\alpha_{g}+1=\dim_{F}(\overline{A}_{g})+1$.
Again this means that $f$ will vanish if we evaluate any of these
sets by semisimple elements. It follows that in each one of these
$s$-sets at least one of the evaluations is radical. Since $s >
n_{A}-1$, the polynomial $f$ vanishes on such evaluations as well
and hence $f$ is a $G$-graded identity of $A$. Contradiction.
\end{proof}

In the next examples we show that the Kemer points of $A$ may be
quite far from $G-\Par(A)$.

\begin{example}

Let $A$ be $G$-simple and let $A(n)=A\times A \times \cdots \times
A$ ($n$-times). Clearly, $\id_{G}(A(n))=\id_{G}(A)$ and hence $A$
and $A(n)$ have the same Kemer points. On the other hand
$G-\Par(A(n))$ increases with $n$.

\end{example}

The next construction (see [BR], Example $4.50$) shows that also the
nilpotency index may increase indefinitely whereas that Kemer points
remain unchanged.

Let $B$ be any finite dimensional $G$-grade algebra and let
$B^{'}=\overline{B}*\{x_{g_1},\dots,x_{g_n}\}$ be the algebra of
$G$-graded polynomials in the graded variables $\{x_{g_i}\}$ with
coefficients in $\overline{B}$, the semisimple component of $B$.
The number of $g$-variables that we take is at least the dimension
of the $g$-component of $J(B)$. Let $I_{1}$ be the ideal of $B$
generated by all evaluations of polynomials of $\id_{G}(B)$ on
$B^{'}$ and let $I_{2}$ be the ideal generated by all variables
$\{x_{g_i}\}$. Consider the algebra $\widehat{B}_{u}=B/(I_{1} +
I_{2}^{u})$.

\begin{proposition}\label{Control on nilpotency index}

\begin{enumerate}

\item

$\id_{G}(\widehat{B}_{u})= \id_{G}(B)$ whenever
$u \geq n_{B}$ ($n_{B}$ denotes the nilpotency index of $B$).
In particular $\widehat{B}_{u}$ and $B$ have the same Kemer points.

\item

$\widehat{B}_{u}$ is finite dimensional.

\item

The nilpotency index of $\widehat{B}_{u}$ is $u$.

\end{enumerate}

\end{proposition}

\begin{proof}

Note that by the definition of $\widehat{B}_{u}$,
$\id_{G}(\widehat{B}_{u}) \supseteq \id_{G}(B)$. On the other hand
there is a $G$-graded surjection
$\id_{G}(\widehat{B}_{u}) \longrightarrow \id_{G}(B)$ which maps
the variables $\{x_{g_i}\}$ into graded elements of $J(B)$ where
$\overline{B}$ is mapped isomorphically. This shows (1). To see
(2) observe that any element in $\widehat{B}_{u}$ is represented
by a sum of the form $b_{1}z_1b_2z_2 \cdots b_{j}z_jb_{j+1}$ where
$j < u$, $b_{i}\in \overline{B}$ and $z_i \in \{x_{g_i}\}$.
Clearly the subspace spanned by monomials for a given
configuration of the $z_i$'s (and arbitrary $b_i$'s) has finite
dimension. On the other hand the number of different
configurations is finite and so the result follows. The 3rd
statement follows from the fact that the product of less than
$u-1$ variables is non zero in $\widehat{B}_{u}$.
\end{proof}

In view of the examples above, in order to establish a precise
relation between Kemer points of a finite dimensional $G$-graded
algebra $A$ and its structure we need to find appropriate finite
dimensional algebras which will serve as \underline{minimal
models} for a given Kemer point. We start with the decomposition
of a $G$-graded finite dimensional algebra into the product of
subdirectly irreducible components.

\begin{definition}

A $G$-graded finite dimensional algebra $A$ is said to be
\textit{subdirectly irreducible}, if there are no non-trivial
$G$-graded ideals $I$ and $J$ of $A$, such that $I\cap J=\{0\}$.

\end{definition}

\begin{lemma}\label{decomposition into irreducible}

Let $A$ be a finite dimensional $G$-graded algebra over $F$. Then $A$ is
$G$-graded \textit{PI}-equivalent to a direct product $C_1\times \cdots \times C_n$
of finite dimensional $G$-graded algebras, each subdirectly irreducible. Furthermore for every
$i=1,\ldots,n$, $dim_{F}(C_i)\leq dim_{F}(A)$ and the number of $G$-simple components in $C_i$
is bounded by the number such components in $A$.
\end{lemma}

\begin{proof}

The proof is identical to the proof in the ungraded case. If $A$
is not subdirectly irreducible we can find non-trivial $G$-graded
ideals $I$ and $J$ such that  $I\cap J=\{0\}$. Note that $A/I\times
A/J$ is \textit{PI}-equivalent to $A$. Since $\dim_{F}(A/I)$ and
$\dim_{F}(A/J)$ are strictly smaller than $\dim_{F}(A)$ the result
follows by induction.
\end{proof}

The next condition is

\begin{definition}

We say that a finite dimensional $G$-graded algebra $A$ is
\textit{full} with respect to a $G$-graded multilinear polynomial
$f$, if there exists a non-vanishing evaluation such that every
$G$-simple component is represented (among the semisimple
substitutions).
A finite dimensional $G$-graded algebra $A$ is said to be full if it
is full with respect to some $G$-graded polynomial $f$.
\end{definition}

We wish to show that any finite dimensional algebra may be decomposed (up to
\textit{PI}-equivalence) into the direct product of full algebras. Algebras without identity are treated
separately.

\begin{lemma}
Let $A$ be a $G$-graded, subdirectly irreducible which is not full.
\begin{enumerate}
\item
If $A$ has an identity then it is \textit{PI}-equivalent to a direct
product of finite dimensional algebras, each having fewer
$G$-simple components.

\item
If $A$ has no identity then it is \textit{PI}-equivalent to a direct
product of finite dimensional algebras, each having either fewer
$G$-simple components than $A$ or else it has an identity element
and the same number of $G$-simple components as $A$.

\end{enumerate}

\end{lemma}

The lemma above together with Lemma \ref{decomposition into irreducible} yields:

\begin{corollary}

Every finite dimensional $G$-graded algebra $A$ is \textit{PI}-equivalent
to a direct product of full, subdirectly irreducible
finite dimensional algebras.
\end{corollary}

\begin{proof}(of the lemma) The proof is similar to the proof in \cite{BR}.
Assume first that $A$ has identity.
Consider the decompositions mentioned above $A \cong
\overline{A}\oplus J$ and  $\overline{A}\cong A_{1}\times
A_{2}\times\cdots\times A_{q}$ ($A_{i}$ are $G$-simple algebras).
Let $e_{i}$ denote the identity element of $A_{i}$ and consider the
decomposition $A \cong \oplus_{i,j=1}^{q}e_{i}Ae_{j}$. By
assumption we have that $e_{i_1}Ae_{i_2}\cdots
e_{i_q-1}Ae_{i_q}=e_{i_1}Je_{i_2}\cdots e_{i_q-1}Je_{i_q}=0$
whenever $i_1,\ldots,i_q $ are distinct.

Consider the commutative algebra
$H=F[\lambda_1,\ldots,\lambda_q]/I$ where $I$ is the (ungraded)
ideal generated by $\lambda_{i}^{2}-\lambda_{i}$ and
$\lambda_1\cdots\lambda_q$. Then, if we write $\widetilde{e}_{i}$
for the image of $\lambda_i$, we have
$\widetilde{e}_{i}^{2}=\widetilde{e}_{i}$ and
$\widetilde{e}_{1}\cdots\widetilde{e}_{q}=0$. Consider the algebra
$A\otimes_{F}H$, $G$-graded via $A$. Let $\widetilde{A}$ be the
$G$-graded subalgebra generated by all $e_{i}Ae_{j}\otimes
\widetilde{e}_{i}\widetilde{e}_{j}$ for every $1 \leq i,j \leq q$.
We claim that the graded algebras $A$ and $\widetilde{A}$ are
\textit{PI}-equivalent as $G$-graded algebras: Clearly $\id_{G}(A)\subseteq
\id_{G}(A\otimes_{F}H) \subseteq \id_{G}(\widetilde{A})$, so it
suffices to prove that any graded non-identity $f$ of $A$ is also
a non-identity of $\widetilde{A}$. Clearly, we may assume $f$ is
multilinear and strongly homogeneous. Note: $\widetilde{A}$ is
graded by the number of distinct $\widetilde{e}_{i}$ appearing in
the tensor product of an element. In evaluating $f$ on $A$ it
suffices to consider specializations of the form $x_{g}\longmapsto
v_{g}$ where $v_{g}\in e_{i_{k}}Ae_{i_{k+1}}$. In order to have
$v_{g_{1}}\cdots v_{g_{n}}\neq 0$, the set of indices $i_{k}$
appearing must contain at most $q-1$ distinct elements, so
$e_{i_1}\cdots e_{i_{n+1}}\neq 0$. Then $f(v_{g_1}\otimes
\widetilde{e}_{i_1},\ldots, v_{g_n}\otimes
\widetilde{e}_{i_n})=f(v_{g_1},\ldots, v_{g_n})\otimes
\widetilde{e}_{i_1}\cdots \widetilde{e}_{i_n}\neq 0$ so $f$ is not
in $\id_{G}(\widetilde{A})$. We conclude that
$\widetilde{A}\sim_{PI}A$, as claimed. Now we claim that $\widetilde{A}$
can be decomposed into direct product of
$G$-graded algebras with fewer $G$-simple components.
To see this let $I_{j}=<e_{j}\otimes \widetilde{e}_{j}> \lhd
\widetilde{A}$. Note that by Theorem \ref{BSZ}, the element $e_{j}$ is homogeneous and
hence $I_{j}$ is $G$-graded. We see that

$$\bigcap_{j=1}^{q}I_{j}=(1\otimes \widetilde{e}_{1}\cdots 1\otimes
\widetilde{e}_{q}(\bigcap_{j=1}^{q}I_{j})=(1\otimes
\widetilde{e}_{1}\cdots
\widetilde{e}_{q})(\bigcap_{j=1}^{q}I_{j})=\{0\}$$

so $\widetilde{A}$ is subdirectly reducible to the direct product of
$\widetilde{A}/I_{j}$. Furthermore, each component $\widetilde{A}/I_{j}$
has less than $q$ $G$-simple components since
we eliminated the idempotent corresponding to the $j$'th $G$-simple
component. This completes the proof of the first part of the lemma.

Consider now the case where the algebra $A$ has no identity. There
we have $A \cong \oplus_{i,j=0}^{q}e_{i}Ae_{j}$ where
$e_{0}=1-(e_1+ \cdots + e_q)$ (here the element $1$ and hence
$e_{0}$ are given degree $e\in G$). We proceed as above but now
with $q+1$ idempotents, variables, etc. We see as above that
$\widetilde{A}/I_{j}$ will have less than $q$ $G$-simple
components if $1\leq j \leq q$ whereas $\widetilde{A}/I_{0}$ will
have an identity and $q$ $G$-simple components. This completes the
proof of the lemma.
\end{proof}

\begin{remark}
Note that in the decomposition above the nilpotency index of the
components in the direct product is bounded by the nilpotency index
of $A$.
\end{remark}

Now we come to the definition of \textit{PI}-minimal.

\begin{definition}

We say that a finite dimensional $G$-graded algebra $A$ is
\textit{PI}-minimal if $G-\Par_{A}$ is minimal (with respect to
the partial order defined above) among all finite dimensional
$G$-graded algebras which are \textit{PI}-equivalent to $A$.

\end{definition}

\begin{definition}

A finite dimensional $G$-graded algebra $A$ is said to be
\textit{PI}-\textit{basic} (or just \textit{basic}) if it is \textit{PI}-minimal, full and
subdirectly irreducible.

\end{definition}

\begin{proposition}\label{basic algebras}

Every finite dimensional $G$-graded algebra $A$ is
\textit{PI}-equivalent to the direct product of a finite number of
$G$-graded \textit{PI}-basic algebras.

\end{proposition}

In the next three sections we show that any basic $A$ algebra has
a Kemer set which consists of a unique point $(\alpha,
s(\alpha))$. Moreover, $(\alpha, s(\alpha))=G-\Par_{A}$.
\end{section}

\begin{section}{Kemer's Lemma $1$}\label{Section: Kemer's Lemma $1$}

The task in this section is to show that if $A$ is subdirectly
irreducible and full, there is a point $\alpha \in E_{0}(A)$ with
$\alpha = (d_{A,g_{1}},\ldots,d_{A,g_{r}})$. In particular,
$E_{0}(A)$ consists of a unique point.

\begin{proposition}\label{unique point}

Let $A$ be a finite dimensional algebra, $G$-graded, full and
subdirectly irreducible. Then there is an extremal point $\alpha$
in $E_{0}(A)$ with $\alpha_{g} = \dim_{F}(\overline{A}_{g})$ for
every $g$ in $G$. In particular, the extremal point is unique.

\end{proposition}

\begin{proof}

Note that the uniqueness follows from Lemma \ref{easy direction}
since $\alpha \preceq
(\dim_{F}(\overline{A}_{g_1},\ldots,\dim_{F}(\overline{A}_{g_r}))$
for every extremal point $\alpha$.

For the proof we need to show that for an arbitrary large integer
$\nu$ there exists a non-identity $f$ such that for every $g\in G$
there are $\nu$ folds of alternating $g$-sets of cardinality
$\dim_{F}(\overline{A}_{g})$.

Since the algebra $A$ is full, there is a multilinear, $G$-graded
polynomial $f(x_{1,g_{i_1}},\ldots,x_{q,g_{i_q}},
\overrightarrow{y})$, which does not vanish upon an evaluation of
the form $x_{j,g_{i_j}}= \overline{x}_{j,g_{i_j}} \in
\overline{A}_{j,g_{i_j}}$, $j=1,\ldots,q$ and the variables of
$\overrightarrow{y}$ get homogeneous values in $A$. The idea is to
produce polynomials $\widehat{f}$ in the $T$-ideal generated by
$f$ which will remain non-identities of $A$ and will reach
eventually the desired form. The way one checks that the
polynomials $\widehat{f}$'s are non-identities is by presenting
suitable evaluations on which they do not vanish. Let us
reformulate what we need in the following lemma.

\begin{lemma}[Kemer's lemma $1$ for $G$-graded algebras] \label{full-folds}

Let $A$ be a finite dimensional algebra, $G$-graded, subdirectly
irreducible and full with respect to a polynomial $f$. Then for
any integer $\nu$ there exists a non-identity  $f^{'}$ in the
$T$-ideal generated by $f$ with $\nu$-folds of alternating
$g$-sets of cardinality $\dim_{F}(\overline{A}_{g})$ for every $g$
in $G$.

\end{lemma}

\begin{proof}

Let $f_{0}$ be the polynomial obtained from $f$ by multiplying (on
the left say) each one of the variables
$x_{1,g_{i_1}},\ldots,x_{q,g_{i_q}}$ by $e$-homogeneous variables
$z_{1,e},\ldots,z_{q,e}$ respectively. Note that the polynomial obtained is a non
($G$-graded) identity since the variables $z_{i,e}$'s may be evaluated by
the (degree $e$) elements ${1_{\overline{A}_i}}$'s where

$${1_{\overline{A}_i}} = 1\otimes E^{i}_{1,1}+\cdots + 1\otimes E^{i}_{k_i,k_i}$$

Here we use the notation of Theorem \ref{BSZ}, $\overline{A}_i= F^{c_i}H_i \otimes M_{k_i}(F)$.

Applying linearity there exists a
non-zero evaluation where the variables $z_{1,e},\ldots,z_{q,e}$
take values of the form
$1\otimes E^{1}_{j_1,j_1},\ldots,1\otimes E^{q}_{j_q,j_q}$ where $1\leq j_i \leq k_i$ for $i=1,\ldots,q$.

Our aim is to replace each one of the variables
$z_{1,e},\ldots,z_{q,e}$ by $G$-graded polynomials
$Z_{1},\ldots,Z_{q}$ such that:

\begin{enumerate}

\item

For every $i=1,\ldots,q$ and for every $g\in G$, the polynomial
$Z_{i}$ is alternating in $\nu$-folds of $g$-sets of cardinality
$\dim_{F}(\overline{A}_{i})_{g}$.

\item

For every $i=1,\ldots,q$ the polynomial $Z_{i}$ assumes the value
$1\otimes E^{i}_{j_i,j_i}$.

\end{enumerate}

Once this is accomplished, we complete the construction by
alternating the $g$-sets which come from different $Z_{i}$'s.
Clearly, the polynomial $f^{'}$ obtained

\begin{enumerate}

\item

is a nonidentity since any non-trivial alternation of the
evaluated variables (as described above) vanishes.

\item

for every $g\in G$,  $f^{'}$ has $\nu$-folds of alternating
$g$-sets of cardinality $\dim_{F}(\overline{A}_{g})$.

\end{enumerate}

We now show how to construct the $G$-graded polynomials $Z_{i}$.

In order to simplify the notation we put $\widehat{A}=\overline{A}_{i}$
($=F^{c}H \otimes M_{k}(F)$) and
$\widehat{A}=\oplus_{g\in G}\widehat{A}_{g}$ where
$\overline{A}_{i}$ is the i-th $G$-simple component.

Fix a product of the $k^{2}$ different elementary matrices
$E_{i,j}$ in $M_{k}(F)$ with value $E_{1,1}$ (it is not difficult to show
the such a product exists). For each $h\in H$ we
consider the basis elements of $\widehat{A}$ of the form
${b_h\otimes E_{i,j}}$ where $1\leq i,j\leq k$. If we multiply
these elements (in view of the $E_{i,j}$'s ) in the same order as above we obtain
$b_{h}^{k^{2}}\otimes E_{1,1}$. Observe that since $b_{h}$ is
invertible in $F^{f}H$, the element $b_{h}^{k^{2}}\otimes E_{1,1}$
is not zero. Repeating this process for every $h\in H$ and multiplying all
together we obtain a
nonzero product of all basis elements ${b_h\otimes E_{i,j}}$ where
$1\leq i,j\leq k$ and $h\in H$.

Note that we obtained an homogeneous element of the form $\lambda
b_{h}\otimes E_{1,1}$ where $\lambda\in F^{*}$ and $h\in H$.
Finally we may multiply the entire product by $(\lambda
b_{h})^{-1}\otimes E_{1,1}$  and get $1\otimes E_{1,1}$. Consider
now the graded monomial consisting with $\ord(H)\cdot k^{2}+1$
graded variables which may be evaluated by the product above. We
extend this monomial by inserting $\ord(H)\cdot k^{2}+1$
$e$-variables $y_{i,e}$ bordering each one of the basis elements
(there is no need to border the extra $h^{-1}$-variable). Clearly,
the variables $y_{i,e}$ may be substituted by elements of the form
$1\otimes E_{i,i}\in \widehat{A}_{e}$.

\begin{remark}\label{different borderings}

Note that for any $(i,j)$ the number of basis elements
${b_h\otimes E_{i,j}}$ which are bordered by the pair $(1\otimes
E_{i,i}, 1\otimes E_{j,j})$ is precisely $\ord(H)$. The key
observation here is that basis elements which are bordered by the
same pair belong to \textit{different} homogeneous components of
$\widehat{A}$ (and hence there is no need to alternate the
corresponding variables).

\end{remark}

Now, we can complete the proof of Lemma \ref{full-folds}. By
Remark \ref{different borderings}, if we alternate the
$g$-variables that correspond (in the evaluation above) to the
$g$-basis of any $G$-simple component, we obtain a non-identity
(indeed, the evaluations which correspond to non-trivial permutations will
vanish since the borderings of these $g$-elements are different).
Furthermore, alternating $g$-variables which correspond to
$g$-bases of different $G$-simple components again yields a
non-identity since (again) the evaluations which correspond to
non-trivial permutations will vanish (here we are multiplying two
central idempotents of different $G$-simple components). This
completes the proof of Lemma \ref{full-folds} and of Proposition
\ref{unique point}
\end{proof} \end{proof}

\end{section}

\begin{section}{Kemer's Lemma $2$}\label{Section: Kemer's Lemma $2$}

In this section we prove the $G$-graded versions of Kemer's Lemma
$2$. Before stating the precise statement we need an additional
reduction which enables us to control the number of radical
evaluations in certain non-identities.

Let $f$ be a multilinear, graded polynomial which is not in
$\id_{G}(A)$. Clearly, any non-zero evaluation cannot have more
than $n_{A}-1$ radical evaluations.

\begin{lemma}

Let $A$ be a finite dimensional algebra, $G$-graded. Let $(\alpha,
s_{A}(\alpha))$ be a Kemer point of $A$. Then  $s_{A}(\alpha) \leq
n_{A}-1$.

\end{lemma}

\begin{proof}

By the definition of the parameter $s=s_{A}(\alpha)$ we know that
for arbitrary large $\nu$ there exist multilinear graded
polynomials, not in $\id_{G}(A)$, such that have $\nu$-folds of
alternating $g$-sets (small) of cardinality
$d_{g}=\dim_{F}(\overline{A}_{g})$ and $s_{A}(\alpha)$ (big) sets
of cardinality $d_{g}+1$. It follows that an alternating $g$-set
of cardinality $d_{g}+1$ in a non-identity polynomial must have at
least one radical evaluation. Consequently we cannot have more
than $n_{A}-1$ of such alternating sets, proving the Lemma.
\end{proof}

The next definition and proposition are taken from \cite{BR}.
\begin{definition}

Let $f$ be a multilinear $G$-graded polynomial which is not in
$\id_{G}(A)$. We say that $f$ has property $K$ if $f$ vanishes on
every evaluation with less than $n_{A}-1$ radical substitutions

We say that a finite dimensional $G$-graded algebra $A$ has property
$K$ if it satisfies the property with respect to some non-identity
multilinear polynomial.

\end{definition}

\begin{proposition}\label{property K}

Let $A$ be a \textit{PI}-minimal $G$-graded, finite dimensional
$F$-algebra. Then it has property $K$.

\end{proposition}

\begin{proof}

Assume property $K$ always fails. This means that any multilinear
polynomial which vanishes on less than $n_{A}-1$ radical
evaluations is in $\id_{G}(A)$. Recall (from Proposition \ref{Control on nilpotency
index}) the algebra $\widehat{A}_{u}=A^{'}/(I_{1}+I_{2}^{u})$ where
$A^{'}=\overline{A}\ast F\{x_{1,g_1},\ldots,x_{\nu,g_{\nu}}\}$.
We claim that for $u=n_{A}-1$, $\widehat{A}_{u}$ is \textit{PI}-equivalent to $A$.
Then noting that the nilpotency index of $\widehat{A}_{n_{A}-1}$ is
$n_{A}-1$ we get a contradiction to the minimality of $A$. Clearly,
by construction $\id_{G}(A) \subseteq \widehat{A}_{n_{A}-1}$. For the
converse take a polynomial $f$ which is not in $\id_{G}(A)$. Then by assumption,
there is a non-zero evaluation on $A$ with less than $n_{A}-1$ radical
substitutions. Now, if we replace these radical values by $\{x_{i,g}\}$'s
in $A^{'}$ we get a polynomial (in $\{x_{i,g}\}$'s) which is not
in $I_{1}+I_{2}^{n_{A}-1}$ and hence $f$ is not in $\id_{G}(\widehat{A}_{n_{A}-1})$.
This completes the proof of the proposition.
\end{proof}

Let $A$ be a basic algebra (i.e minimal, full and subdirectly
irreducible). Let $G-\Par_{A}=(d_{g_1},\ldots,d_{g_r}; n_{A}-1)$
where $d_{g}$ is the dimension of the $g$-homogeneous components
of the semisimple part of $A$ and $n_{A}$ the nilpotency index of
$J(A)$. By Proposition \ref{property K} the algebra satisfies
property $K$ with respect to a non-identity polynomial $f$, that
is $f$ does vanish on any evaluation whenever we have less than
$n_{A}-1$ radical substitutions. Furthermore, there is possibly a
different non-identity polynomial $h$ with respect to which $A$ is
full, that is $h$ has a nonzero evaluation which ``visits'' each
one of the $G$-simple components of $\overline{A}$. In order to
proceed we need both properties to be satisfied by the same
polynomial.

We start with two preliminary lemmas which show that these two
properties, namely property $K$ and the property of being full are
``preserved'' in a $T$-ideal.

\begin{lemma}\label{Phoenix}

Let $A$ be a $G$-graded finite dimensional algebra over $F$.

\begin{enumerate}

\item

Let $f$ be a $G$-graded non-identity polynomial, strongly
homogeneous which is $\mu$-fold alternating on $g$-sets of
cardinality $d_{g}=\dim_{F}(\overline{A}_{g})$ for every $g$ in
$G$ (in particular $A$ is full with respect to $f$). If $f^{'}\in
\langle f \rangle $ is a non-identity in the $T$-ideal generated
by $f$, then there exists a non-identity $f^{''}\in \langle
f^{'}\rangle$ which is $\mu$-fold alternating on $g$-sets of
cardinality $d_{g}$ for every $g$ in $G$. In other words, the
property of being $\mu$-fold alternating on $g$-sets of
cardinality $d_{g}$ for every $g$ in $G$ is $A$-Phoenix.

\item

Property $K$ is strictly $A$-Phoenix.

\end{enumerate}

\end{lemma}

\begin{proof}

Let $f$ be a $G$-graded non-identity polynomial, strongly
homogeneous which is $\mu$-fold alternating on $g$-sets of
cardinality $d_{g}$ for every $g$ in $G$ and let $f^{'}$ be a
non-identity in $\langle f \rangle$. In view of Lemma \ref{full-folds},
it is sufficient to show that $A$ is full with respect to $f^{'}$.
Note that by the definition
of $\mu$, for each $g$ in $G$, in at least one alternating $g$-set
the evaluations of the corresponding variables must consist of
semisimple elements of $A$ in any nonzero evaluation of the
polynomial. The result is clear if $f^{'}$ is in the ideal (rather
than the $T$-ideal) generated by $f$. We assume therefore that
$f^{'}$ is obtained from $f$ by substituting variables $x_{g}$'s
by monomials $Z_{g}$'s. Clearly, if one of the evaluations in any
of the variables of $Z_{g}$ is radical, then the value of $Z_{g}$
is radical. Hence in any non-zero evaluation of $f^{'}$ there is
an alternating $g$-set $\Delta_{g}$ of cardinality $d_{g}$ in $f$,
such that the variables in monomials of $f^{'}$ (which correspond
to the variables in $\Delta_{g}$) assume only semisimple values. Furthermore, each
$G$-simple component must be represented in these evaluations for
otherwise we would have a $G$-simple component not represented in
the evaluation of $\Delta_{g}$ and this is impossible. We have
shown that $A$ is full with respect to $f^{'}$. Applying Lemma
\ref{full-folds} we obtain a polynomial $f^{''}\in \langle
f^{'}\rangle \subseteq \langle f \rangle$ with the desired
property. This proves the first part of the lemma.

For the second part note that if $f^{'}$ is a non-identity in the
$T$-ideal generated by $f$ then if $f^{'}$ has less than $n_{A}-1$
radical evaluations then the same is true for $f$ and hence
vanishes. In other word, $f^{'}$ satisfies property $K$.
\end{proof}

\begin{remark}
Note that we are not claiming that the property ``full'' is
Phoenix.
\end{remark}

Now we combine these two properties.

\begin{lemma}\label{combining Full and property K}

Let $A$ be a finite dimensional algebra, $G$-graded which is full,
subdirectly irreducible and satisfying property $K$. Let $f$ be a
non-identity with $\mu$-folds of alternating $g$-sets of
cardinality $d_{g}=\dim_{F}(\overline{A}_{g})$ for every $g$ in
$G$ and let $h$ be a polynomial with respect to which $A$ has
property $K$. Then there is a non-identity in $\langle f \rangle
\cap \langle h \rangle$. Consequently there exists a non-identity
polynomial $\widehat{f}$ which has $\mu$-folds of alternating
$g$-sets of cardinality $d_{g}$ for every $g$ in $G$ and with
respect to which $A$ has property $K$.

\end{lemma}

\begin{proof}
Suppose this is false, that is the intersection is contained in
$\id_{G}(A)$. Consider the ideals $\overline{I}$ and
$\overline{J}$ generated by all evaluations on $A$ of the
polynomials in the $T$-ideals $I=\langle f \rangle$ and $J=\langle
h \rangle$ respectively. Since the ideals $I$ and $J$ are not
contained in $\id_{G}(A)$, the ideals $\overline{I}$ and
$\overline{J}$ are nonzero. On the other hand, by construction,
their intersection is zero and we get a contradiction to the
subdirectly irreducibility of $A$. Take a non-identity $f^{'} \in
I\cap J$. By the first part of Lemma \ref{Phoenix} there is a
non-identity $\widehat{f}\in \langle f^{'} \rangle \subseteq I\cap
J$ which has $\mu$-folds of alternating $g$-sets of cardinality
$d_{g}$ for every $g$ in $G$. By the second part of the Lemma
$\widehat{f}$ has property $K$.
\end{proof}

We can now state and prove Kemer's lemma $2$ for $G$-graded
algebras.

\begin{lemma} [Kemer's lemma $2$ for $G$-graded algebras] \label{Kemer's Lemma 2}

Let $A$ be a finite dimensional $G$-graded algebra. Assume $A$ is
basic. Let $(d_{g_1},\ldots,d_{g_r})$ be the dimensions of the
$g_{i}$-homogeneous components of $\overline{A}$ (the semisimple
part of $A$) and $n_{A}$ be the nilpotency index of $J(A)$. Then for
any integer $\nu$ there exists a $G$-graded, multilinear,
non-identity polynomial $f$ such that for every $g$ in $G$, it has
$\nu$-folds of $g$-sets of alternating variables of cardinality
$d_{g}$ and a total of $n_{A}-1$ sets of variables which are
$g$-homogeneous and of cardinality $d_{g}+1$ for some $g$ in $G$.

\end{lemma}

\begin{note}

Any nonzero evaluation of such $f$ must consists only of
semisimple evaluations in the $\nu$-folds and each one of the big
sets (namely the sets of cardinality $d_{g}+1$)  must have exactly
one radical evaluation.

\end{note}

\begin{proof}

By the preceding Lemma we take a multilinear (strongly
homogeneous) non-identity polynomial $f$, with respect to which $A$ is
full and has property $K$. Let us fix a non-zero evaluation $x_{g}
\longmapsto \widehat{x}_{g}$. We will consider four cases. These correspond to whether $A$
has or does not have an identity and whether $q$ (the number of $G$-simple components) $>1$
or $q=1$.

{\bf Case $(1,1)$} ($A$ has identity and $q>1$).

Fix a nonzero evaluation of $f$ with respect to which $A$ is full
(i.e. ``visits'' in every $G$-simple component) and has precisely
$n_{A}-1$ radical substitutions. Moreover any evaluation with
fewer radical substitutions vanishes. We choose a monomial $X$ in $f$
which does not vanish upon the above evaluation.

Notice that in the monomial $X$, the variables which get semisimple
evaluations from different $G$-simple component must be separated by
variables with radical values.
Let us denote by $w_{1,g_{i_1}},\ldots,w_{n_{a}-1,g_{i_{n_{A}-1}}}$ the variables
which get radical values and by $\widehat{w}_{1,g_{i_1}},\ldots,\widehat{w}_{n_{a}-1,g_{i_{n_{A}-1}}}$
their values ($g_{i_j}$ is the $G$-degree of $w_{j,g_{i_j}})$.

By Theorem \ref{BSZ} (and linearity), each $\widehat{w}_{i,g_i}$ may be
bordered (i.e. multiplied from left and right) by elements of the
form $1\otimes E^{i}_{k_i,k_i}\in A_{i}$ and still giving a nonzero value.
We refer to any two such elements of the form $1\otimes
E^{i}_{k_i,k_i}$ which border a radical evaluation as
\textit{partners}.

\begin{claim}

The elements $1\otimes E^{i}_{k_i,k_i}$ which appear in the
borderings above, represent all the $G$-simple components of
$\overline{A}$.

\end{claim}

Indeed, suppose that the $G$-component $A_{1}$ (say) is not
represented among the $1\otimes E^{i}_{k_i,k_i}$'s. Since our
original evaluation was full there is a variable which is
evaluated by an element $u_g$ of $A_{1}$.
``Moving'' along the monomial $X$ to the left or right of $u_{g}$ we
will hit a bordering value of the form $1\otimes E^{i}_{k_i,k_i}$ before we
hit any radical evaluation. But this is possible only if both
$u_g$ and $1\otimes E^{i}_{k_i,k_i}$ belong to the same $G$-simple
component. This proves the claim.

But we need more: Consider the radical evaluations which are
bordered by pairs of elements $1\otimes E^{i}_{k_i,k_i}, 1\otimes
E^{j}_{k^{'}_j,k^{'}_j}$ that belong to $G$-simple
components $A_{i}$ and $A_{j}$ where $i\neq j$.

\begin{claim}

Every $G$-simple component is represented by one of the elements
in these pairs.

\end{claim}

Again, assume that $A_1$ is not represented among these
pairs. By the preceding claim $A_1$ is represented, so it
must be represented by both partners in each pair it appears. Take
such a pair: $1\otimes E^{1}_{k_1,k_1}, 1\otimes
E^{1}_{k^{'}_1,k^{'}_1}$. Moving along the monomial $X$ to the left of
$1\otimes E^{1}_{k_1,k_1}$ or to the right of $1\otimes
E^{1}_{k^{'}_1,k^{'}_1}$, we will hit a value in a different
$G$-simple component. But before that we must hit a radical
evaluation which is bordered by a pair where one of the partners is
from $A_1$ and the other from a different $G$-simple
component. This contradicts our assumption and hence the claim is
proved.

Let $y_{1,e},\ldots,y_{q,e}$ $e$-variables (each will correspond to a $G$-simple component). For $t=1,\ldots,q$ we choose a variable $w_{j_t,g_{i_{j_t}}}$ whose radical value $\widehat{w}_{j_t,g_{i_{j_t}}}$ is bordered
by partners which

\begin{enumerate}

\item

belong to different $G$-simple components.

\item

one of them is an idempotent in the $t$-th $G$-simple component.

\end{enumerate}

We replace now the variable $w_{j_t,g_{i_{j_t}}}$ by the product $y_{t,e}w_{j_t,g_{i_{j_t}}}$ or
$w_{j_t,g_{i_{j_t}}}y_{t,e}$ (according to the position of the bordering). Clearly we obtained a non-identity.

Applying Lemma \ref{full-folds} we can insert in the
$y_{j,e}$'s suitable graded polynomials and obtain a
$G$-graded polynomial with $\nu$-folds of (small) $g$-sets of alternating
variables where each $g$-set is of cardinality
$\dim(\overline{A}_{g})$.

Consider the variables with radical
evaluations which are bordered by $e$-variables with evaluations
from different $G$-simple components (these include the variables
which are bordered by the $y_{j,e}$). Let $z_{g}$ be such a
variable (assume it is homogeneous of degree $g$). We attach it to
a (small) $g$-alternating set. We claim that if we alternate this set (of
cardinality $d_{g}+1$) we obtain a non-identity. Indeed, all
$g$-variables in the small set are bordered by $e$-variables which
are evaluated with elements from the \textit{same} $G$-simple component
whereas the radical element is bordered with elements of \textit{different}
$G$-simple components. Consequently any non-trivial permutation of
the evaluated monomials vanishes. At this point we have
constructed the desired number of small sets and \textit{some} of the
big sets. We still need to attach the radical variables which are
bordered by $e$-variables from the same $G$-simple component. We
attach them as well to (small) $g$-sets. We claim also here that if we
alternate this set (of cardinality $d_{g}+1$) we obtain a
non-identity. Indeed, any non-trivial permutation represents an
evaluation with fewer radical evaluations in the original
polynomial which must vanish by property $K$. This completes the
proof where $q$, the number of $G$-simple components, is $> 1$.

\smallskip

{\bf Case $(1,2)$} ($A$ has identity and $q=1$). We start with a
non-identity $f$ which satisfies property $K$. Clearly we may
multiply $f$ by a variable $x_{e}$ and get a non-identity (since
$x_{e}$ may be evaluated by $1$). Again by Lemma \ref{full-folds}
we may replace $x_{e}$ by a polynomial $h$ with $\nu$-folds of
$g$-sets of alternating variables of cardinality $d_{g}$. Consider
the polynomial $hf$. We attach the radical variables of $f$ to
some of the small sets in $h$. Any non-trivial permutation
vanishes because $f$ satisfies property $K$. This completes the
proof of the Lemma \ref{Kemer's Lemma 2} in case $A$ has an
identity.

\smallskip

{\bf Case $(2,1)$}. Suppose now $A$ has no identity and $q>1$. The
proof in this case is basically the same as in the case where $A$
has identity. Let $e_{0}=1-1_{A_1}-1_{A_2}-\cdots -1_{A_q}$ and
include $e_{0}$ to the set of elements which border the radical
values $\widehat{w}_{j,g_{i_j}})$. A similar argument shows that
also here every $G$-simple component ($A_{1},\ldots,A_{q}$) is
represented in one of the bordering pairs where the partners are
\textit{different}. The point is one of the partners (among these
pairs) may be $e_{0}$. Now we complete the proof exactly as in
case $(1,1)$.

\smallskip

{\bf Case $(2,2)$}. In order to complete the proof of the lemma we
consider the case where $A$ has no identity and $q=1$. The
argument in this case is somewhat different. For simplicity we
denote by $e_1=1_{A_1}$ and $e_0=1-e_1$. Let
$f(x_{1,g_{i_1}},\ldots,x_{n,g_{n_1}})$ be a non-identity of $A$
which satisfies property $K$ and let
$f(\widehat{x}_{1,g_{i_1}},\ldots,\widehat{x}_{n,g_{i_n}})$ be a
non-zero evaluation. If
$e_{1}f(\widehat{x}_{1,g_{i_1}},\ldots,\widehat{x}_{n,g_{i_n}})
\neq 0$ (or
$f(\widehat{x}_{1,g_{i_1}},\ldots,\widehat{x}_{n,g_{i_n}})e_{1}$)
we proceed as in case $(1,2)$. To treat the remaining case we may
assume that:

\begin{enumerate}
\item
$f$ is a non-identity and satisfies property $K$

\item

$A$ is full with respect to $f$.

\item

$e_{0}f(\widehat{x}_{1,g_{i_1}},\ldots,\widehat{x}_{n,g_{i_n}})e_{0}\neq 0$.

\end{enumerate}

First note that if one of the radical values (say $\widehat{w}_{g}$) in $f(\widehat{x}_{1,g_{i_1}},\ldots,\widehat{x}_{n,g_{i_n}}$) allows a bordering by the pair $(e_{0},e_{1})$ (and remains nonzero), then replacing $w_{g}$ by $w_{g}y_{e}$ where $y_{e}$ is an $e$-variable, yields a non-identity. Invoking Lemma \ref{full-folds} we may replace the variable $y_{e}$ by a $G$-graded polynomial $h$ with $\nu$-folds of alternating (small) $g$-sets of cardinality $\dim_{F}(\overline{A}_{g})=\dim_{F}((A_{1})_{g})$ for every $g$ in $G$. Then we attach the radical variable $w_{g}$ to a suitable small set (same $G$-degree). Clearly, the value of any alternation of this (big) set is zero since the borderings are different. Finally we attach the remaining radical variables to suitable small sets in $h$. Again any alternation vanishes because of property $K$. This settles this case. Obviously, the same holds if the bordering pair above is $(e_{1},e_{0})$. The outcome is that we may assume that all radical values may be bordered by either $(e_{0},e_{0})$ or $(e_{1},e_{1})$.

\begin{claim}
Under the above assumption, all pairs that border radical values are equal.
\end{claim}

Indeed, if we have of both kinds, we must have a radical value which is bordered by a mixed pair since the semisimple variable can be bordered only the pair $(e_{1},e_{1})$.

Now, assume all the bordering pairs of the radical values are $(e_{1},e_{1})$. Since also the semisimple values can be bordered (only) by that pair it follows that the entire value of the polynomial, namely $f(\widehat{x}_{1,g_{i_1}},\ldots,\widehat{x}_{n,g_{i_n}})$, may be multiplied by $(e_{1},e_{1})$ but this case was already taken care of.

The last case to consider, is the case where the all bordering pairs of the radical values are $(e_{0},e_{0})$. Here we use the fact that the polynomial is full (rather than satisfying property $K$ as in previous cases) and replace one of the semisimple variables (say $x_g$) by $x_{g}y_{e}$. Then as above we replace $y_{e}$ by $G$-graded polynomial $h$ with $\nu$-folds of alternating (small) $g$-sets of cardinality $\dim_{F}(\overline{A}_{g})=\dim_{F}((A_{1})_{g})$ for every $g$ in $G$. The point in this case is that we may attach all radical variables to suitable small sets from $h$. Clearly, since the borderings are different ($(e_{0},e_{0})$ for the radical values and $(e_{1},e_{1})$ for the semisimple ones) any non-trivial alternation will vanish. This completes the proof of the Lemma.

\end{proof}

\begin{corollary}

If $A$ is basic then its Kemer set consists of precisely one point
$(\alpha, s(\alpha))=(\alpha_{g_1},\ldots,\alpha_{g_r}; s(\alpha))=
(d_{g_1},\ldots,d_{g_r}; n_{A}-1)$.

\end{corollary}

\begin{corollary}\label{Kemer polynomials of A are Phoenix}
\begin{enumerate}

\item

Let $A$ be a basic algebra and let $f$ be a Kemer polynomial of
$A$ i.e. a Kemer polynomial of its unique Kemer point $(\alpha,
s(\alpha))$. Then it satisfies the $A$-Phoenix property.

\item
More generally:
let $A$ be a finite dimensional algebra $A$ and let $f$ be a Kemer
polynomial of a Kemer point of $A$. Then it satisfies the
$A$-Phoenix property.

\end{enumerate}
\end{corollary}

\begin{proof}

Clearly if $f$ is Kemer then $A$ is Full and satisfies property
$K$ with respect to $f$. The first part of the Corollary now
follows from Lemmas \ref{Phoenix}, \ref{combining Full and
property K} and \ref{Kemer's Lemma 2}.

The second part follows at once from the first.
\end{proof}

\end{section}

\begin{section}{More tools}

In this section we present several concepts and results which will
be essential for the proof of the main theorems. These concepts
are borrowed from the classical \textit{PI}-theory.

\begin{subsection}{Finite generation of the relatively free
algebra}\label{Finite generation of the relatively free algebra} It
is well known that if $\mathcal{W}$ is a relatively free algebra
over a field of characteristic zero which satisfies the Capelli
identity $c_{n}$, then it has basic rank $ < n$ (i.e. any identity
of $\mathcal{W}$ is equivalent to an identity with less than $n$
variables). Indeed, any non-zero multilinear polynomial with $m$
variables, generating an irreducible $S_{m}$-module corresponds to a
Young tableau with strictly less than $n$-rows and hence is
equivalent (via linearization) to a homogeneous polynomial with less
than $n$-variables (see \cite{Kemer1}, section 1). The same holds
for $G$-graded polynomials, i.e. a polynomial with $m_{i}$
$g_{i}$-variables, $i=1,\dots,r$. Here one considers the action of
the group $S_{m_1} \times \cdots \times S_{m_r}$ on the set of
multilinear polynomial with $m_{1}+\cdots+m_{r}=m$ variables (each
Symmetric group acts on the corresponding variables) and shows that
such a polynomial is equivalent to an homogeneous polynomial with
less than $n$ variables of each type (i.e. $< rn)$. This gives:

\begin{corollary} \label{reduction to affine G-grading}

Let $W$ be a $G$-graded algebra as above. Then there exists a
relatively free affine $G$-graded algebra $\mathcal{W}_{affine}$
with $\id_{G}(W)=\id_{G}(\mathcal{W}_{affine})$.

\end{corollary}

\begin{corollary}\label{finite number of variables in Kemer
polynomials} All $G$-graded Kemer polynomial of $W$ are obtained
(via linearization) from Kemer polynomials with a bounded number of
variables.

\end{corollary}

\begin{remark} We could obtain the corollaries above from Berele
and Bergen result (see \cite{BB}, Lemma 1).

\end{remark}

\end{subsection}

\begin{subsection}{The $G$-graded generic algebra}
We start with an alternative description of the relatively free
algebra $\mathcal{A}=F\langle X_{G} = \cup X_{g}\rangle/
\id_{G}(A)$ of a finite dimensional algebra $A$. Note that by the virtue of
Corollary \ref{reduction to affine G-grading} we may (and will) assume that the set of
$g$-variables is finite, say $m_{g}$, for every
$g\in G$.

Let $\{b_{1,g},b_{2,g},\ldots,b_{t_g,g}\}$ be a basis of the
$g$-component of $A$ and let $\Lambda_{G} = \{\lambda_{i,j,g} \mid
i=1,\ldots,m$, $j=1,\ldots,t_g$ , $g \in G$ $\}$ be a set of
commuting variables which centralize the elements of $A$. For any
$g$ in $G$, consider the elements

$y_{i,g}= \sum_{j} b_{j,g}\lambda_{i,j,g}$ for $i=1,\ldots,m_{g}$

and consider the subalgebra $\widetilde{\mathcal{A}}$ they
generate in the polynomial algebra $A[\Lambda_{G}]$. The proof of
the following lemma is identical to the proof in \cite{BR},
section 3.3.1 and is omitted.

\begin{lemma} \label{generic}

The map $\pi: \mathcal{A} \longrightarrow \widetilde{\mathcal{A}}$
defined by $\pi(x_{i,g})= y_{i,g}$ is a $G$-graded isomorphism.

\end{lemma}

In particular, the relatively free algebra of a finite dimensional
algebra $A$ is representable i.e. it can be $G$-graded embedded in a
finite dimensional algebra. The next claim is well known.

\begin{claim}
Any $G$-graded finite dimensional algebra $A$ over a field $K$ can
be $G$-embedded in a $G$-graded matrix algebra.
\end{claim}

\begin{proof}

Let $n=\dim_{K}(A)$ and let $M=\End_{K}(A) \cong M_{n}(K)$ be the
algebra of all endomorphisms of $A$. We may introduce a
$G$-grading on $\End_{K}(A)$ by setting $M_{g}= \{\varphi \in M$
such that $\varphi(A_{h})\subseteq A_{gh}\}$. Let us show that any
endomorphism of $A$ can be written as a sum of homogeneous
elements $\varphi_{g}$. Indeed, if $\varphi$ is in $M$ and $h \in
G$ we define $\varphi_{g}$ on $A_{h}$ by $\varphi_{g}=P_{gh} \circ
\varphi$ where $P_{gh}$ is the projection of $A$ onto $A_{gh}$.
Taking $a_{h} \in A_{h}$ we have $\oplus_{g} \varphi_{g}(a_{h})=
\oplus_{g} P_{gh} \circ \varphi (a_{h})= \varphi (a_{h})$. Since
this is for every $h$ in $G$ the result follows.
\end{proof}

\end{subsection}

\begin{subsection}{Shirshov (essential) base}

For the reader convenience we recall the definition from classical
\textit{PI}-theory (i.e. ungraded).

\begin{definition}

Let $W$ be an affine \textit{PI}-algebra over $F$. Let
$\{a_1,\ldots,a_s\}$ be a set of generators of $W$. Let $m$ be a
positive integer and let $Y$ be the set of all words in
$\{a_1,\ldots,a_s\}$ of length $\leq m$. We say that $W$ has
\textit{Shirshov} base of length $m$ and of height $h$ if elements
of the form $y_{i_1}^{k_1}\cdots y_{i_l}^{k_l}$ where $y_{i_i} \in
Y$ and $l\leq h$, span $W$ as a vector space over $F$.

\end{definition}

\begin{theorem}

If $W$ is an affine \textit{PI}-algebra, then it has a Shirshov
base for some $m$ and $h$. More precisely, suppose $W$ is
generated by a set of elements of cardinality $s$ and suppose it
has \textit{PI}-degree $m$ (i.e. there exists an identity of
degree $m$ and $m$ is minimal) then $W$ has a Shirshov base of
length $m$ and of height $h$ where $h=h(m,s)$.

\end{theorem}

In fact we will need a weaker condition (see \cite{BBL})

\begin{definition}

Let $W$ be an affine \textit{PI}-algebra. We say that a set $Y$ as
above is an \textit{essential Shirshov} base of $W$  (of length
$m$ and of height $h$) if there exists a finite set $D(W)$ such
that the elements of the form $d_{i_1}y_{i_1}^{k_1}d_{i_2} \cdots
d_{i_l}y_{i_l}^{k_l}d_{i_{l+1}}$ where $d_{i_j}\in D(W)$,
$y_{i_j}\in Y$ and $l\leq h$ span $W$.
\end{definition}

An essential Shirshov's base gives

\begin{theorem}\label{finite module}

Let $C$ be a commutative ring and let $W=C\langle
\{a_1,\ldots,a_s\}\rangle$ be an affine algebra over $C$. If $W$ has
an essential Shirshov base (in particular, if $W$ has a Shirshov
base) whose elements are integral over $C$, then it is a finite
module over $C$.
\end{theorem}

Returning to $G$ graded algebras we have

\begin{proposition}\label{Essential Shirshov base in the
e-component}

Let $W$ be an affine, \textit{PI}, $G$-graded algebra. Then it has
an essential $G$-graded Shirshov base of elements in $W_{e}$.

\end{proposition}

\begin{proof}

$W$ is affine so it is generated by a finite set of elements
$\{a_1,\ldots,a_s\}$ which can be assumed to be homogeneous. We
form the set $Y$ of words of length $\leq m $ in the $a$'s. Then
$Y$ provides a Shirshov base of $W$. Now each element $y$ of $Y$
corresponds to an homogeneous component, say $g$. Hence, raised to
the order of $g$ in $G$ it represents an element in $W_{e}$. Let
$Y_{e}$ be the subset of $W_{e}$ consisting of elements
$y^{\ord(g)}$ where $y\in Y$ of degree $g$ and let $D(W)$ be the
set consisting of all elements of the form
$(1,y,y^{2},\ldots,y^{\ord(g)-1})$. Clearly, $Y_{e}$ is an
essential Shirshov base of $W$.
\end{proof}

\end{subsection}

\begin{subsection}{The trace ring}

Let $A$ be a finite dimensional $G$-graded algebra over $F$ and
let $\mathcal{A}$ be the corresponding relatively free algebra. By
Lemma \ref{generic}, $\mathcal{A}$ is representable, i.e. can be
embedded (as a $G$-graded algebra) into a matrix algebra $M$ over
a suitable field $K$. For every element $x_{e}\in \mathcal{A}_{e}$
(viewed in $M$) we consider its trace $\Tr(x_{e})\in K$. We denote
by $R_{e}=F[\{\Tr(x_{s,e})\}]$ the $F$-algebra generated by the
trace elements of $\mathcal{A}_{e}$. Note that $R_{e}$ centralizes
$\mathcal{A}$ and hence we may consider the extension
$\mathcal{A}_{R_{e}} = R_{e}\otimes_{F} \mathcal{A}$. We refer to
$\mathcal{A}_{R_{e}}$ as the extension of $\mathcal{A}$ by traces
(of $\mathcal{A}_{e}$). In particular we may consider
$(\mathcal{A}_{e})_{R_{e}}$, namely the extension of
$\mathcal{A}_{e}$ by traces.

\begin{remark}

$\mathcal{A}_{e}=0$ if and only if $A_{e}=0$. In that case
$\mathcal{A}_{R_{e}}= \mathcal{A}$.

\end{remark}

\begin{lemma}\label{traces}

The algebras $\mathcal{A}_{R_{e}}$ and $(\mathcal{A}_{e})_{R_{e}}$
are finite modules over $R_{e}$.

\end{lemma}

\begin{proof}

By construction any element in $(\mathcal{A}_{e})_{R_{e}}$ is
integral over $R_{e}$ and hence it has a Shirshov base consisting
of elements which are integral over $R_{e}$. It follows by Theorem
\ref{finite module} that $(\mathcal{A}_{e})_{R_{e}}$ is a finite
module over $R_{e}$. Now, as noted above, since $G$ is finite,
$\mathcal{A}_{R_{e}}$ has an essential Shirshov base ($\subset
(\mathcal{A}_{e})_{R_{e}}$) whose elements are integral over
$R_{e}$. Applying Theorem \ref{finite module} the result follows.
\end{proof}

\end{subsection}

\end{section}

\begin{section}{Kemer polynomials, emulation of traces and representability}\label{Kemer polynomials, emulation of traces and representability}

Let $A$ be a basic $G$-graded algebra and let $\mathcal{A}$ be the corresponding relatively free algebra.

\begin{lemma} \label{representable}

Let $I$ be a $G$-graded $T$-ideal of $\mathcal{A}$ which is closed
under traces (i.e. traces multiplication). Then $\mathcal{A}/I$ is
representable. (This is abuse of language: we should have said $I$
is the $G$-graded ideal of $\mathcal{A}$ generated by the
evaluations (on $\mathcal{A}$) of a $G$-graded $T$-ideal).

\end{lemma}

\begin{proof}
Indeed, by Lemma \ref{traces}, $\mathcal{A}_{R_{e}}$ is a finite
module over $R_{e}$. Hence $\mathcal{A}_{R_{e}}/I_{R_{e}}$ is a
finite module as well. Our assumption on $I$ says $I_{R_{e}}=I$
and hence $\mathcal{A}_{R_{e}}/I$ is a finite module over $R_{e}$.
Now, $R_{e}$ is a commutative Noetherian ring and hence applying
\cite{Bei} we have that $\mathcal{A}_{R_{e}}/I$ is representable.
Since $\mathcal{A}/I \subseteq \mathcal{A}_{R_{e}}/I$,
$\mathcal{A}/I$ is representable as well.
\end{proof}

A key property of Kemer polynomials is \textit{emulation of traces}.
This implies that $T$-ideals generated by Kemer polynomials are
closed under traces. Here is the precise statement.

\begin{proposition}\label{product}

Let $A=A_{1}\times \cdots\times A_{u}$ where the $A_{i}$'s are
basic, $G$-graded algebras. Here $A$ is $G$-graded in the obvious
way. Assume the algebras $A_{i}$ have the same Kemer point. Let
$\mathcal{A}_{i}$ be the relatively free algebra of $A_{i}$ and
set $\widehat{A}= \mathcal{A}_{1}\times\cdots\times
\mathcal{A}_{u}$. Let $S$ be the $T$-ideal generated by a set of
Kemer polynomials of some of the $A_{i}$'s and let
$S_{\widehat{A}}$ be the ideal of $\widehat{A}$ generated by all
evaluation of $S$ on $\widehat{A}$. Then $S_{\widehat{A}}$ is
closed under traces (of $\widehat{A}_{e}$).

\end{proposition}

\begin{proof}

Let $(\alpha_{g_1},\ldots,\alpha_{g_r}, n_{A}-1)$ be the Kemer
point which corresponds to $A$. Recall that if $\alpha_{g_1}=0$
(i.e. $\overline{A}_{e}$, the semisimple part of $e$-component of
$A$ is zero) then $R_{e}=F$ and the proposition is clear. We
assume therefore that $\overline{A}_{e}\neq 0$. Let $z_{e}$ be in
$\mathcal{A}_{e}$ and $f$ in $S$. We need to show that
$\Tr(z_{e})f$ evaluated on $\mathcal{A}$ is in $S_{\widehat{A}}$.
Clearly, we may assume that $f$ is a Kemer polynomial of $A_{1}$.
To simplify the notation we put $d=\alpha_{g_1}$ and write
$f=f(x_{e,1},\ldots,x_{e,d}, \overrightarrow{y})$ where the
variables $x_{e,i}$'s alternate. Let us recall the following
important result on alternating (ungraded) polynomials from
\cite{BR}:

\begin{theorem}[See \cite{BR}, Theorem J]\label{theorem J}

Suppose $B\subseteq M_{n}(K)$ an algebra over $K$, and let $V$ be
a $t$-dimensional $F$-subspace of $M_{n}(K)$ with a base
$a_{1},\dots,a_{t}$ of elements of $B$. Let $f(x_{1},\dots,x_{t};
\overrightarrow{y})$ be an alternating polynomial in the $x$'s. If
$T$ is a $C$-linear map ($C=Z(B)$) $T:V \rightarrow V$, then

$$\Tr(T)f(a_{1},\dots,a_t; \overrightarrow{b})= \sum_{k=1}^{t}f(a_{1},\dots,a_{k-1},Ta_{k},a_{k+1},\dots,a_{t};
\overrightarrow{b})$$

\end{theorem}

First note that the same result (with the same proof) holds for a $G$-graded polynomial $f$ where the $x$'s are
$e$-variables and the space $V$ is contained in the $e$-component of $B$. For our
purposes we consider $\mathcal{A}_{1}$ to be the relatively free algebra of $A_{1}$. Extending scalars to $K$
we have $B=K\otimes \mathcal{A}_{1}$. Then we take $V$ to be the $e$-component
of the semisimple part of $B$. The key observation here is that
since $f$ is a Kemer polynomial, on any nonzero evaluation,
the variables $x_{e,1},\ldots,x_{e,d}$ may assume only values
which form a basis of $V$ and hence the result follows from $G$-graded
version of Theorem \ref{theorem J}. This completes the proof of Theorem \ref{product}.
\end{proof}

Combining Lemma \ref{representable} and Proposition \ref{product}
we obtain:

\begin{corollary}\label{representability}

With the above notation $\widehat{A}/S_{\widehat{A}}$ is
representable.

\end{corollary}

\end{section}

\begin{section}{$\Gamma$-Phoenix property}

Let $A$ be a finite dimensional $G$-graded algebra. Recall that by
Proposition \ref{basic algebras} $A$ is \textit{PI}-equivalent to
a direct product of basic algebras.

Let $A \sim A_{1}\times \cdots\times A_{s}$ where $A_{i}$ are
basic. For each $A_{i}$ we consider its Kemer point
$(\alpha_{A_i}, s(\alpha_{A_i})) = ((d_{A_i, g_1},d_{A_i,
g_2},\ldots,d_{A_i, g_r}), n_{A_i}-1)$. Let $(\alpha_{A_1},
s(\alpha_{A_1})),\dots, (\alpha_{A_t}, s(\alpha_{A_t}))$ be the
Kemer points which are maximal among the Kemer points
$(\alpha_{A_1}, s(\alpha_{A_1})),\dots, (\alpha_{A_s},
s(\alpha_{A_s}))$ (after renumbering if necessary).

\begin{proposition}

$\Kemer(A)= \cup_{1\leq i\leq t}\Kemer(A_{i})$. Furthermore, a
polynomial $f$ is Kemer of $A$ if and only if is Kemer of one of
the $A_{i}, i=1,\ldots,t$.
\end{proposition}

\begin{proof}

Since for any $i=1,\dots,t$, $\id_{G}(A_{i}) \supseteq
\id_{G}(A)$, there exists a Kemer point $(\alpha, s(\alpha))$ of
$A$ with $(\alpha_{A_i}, s(\alpha_{A_i})) \preceq (\alpha,
s(\alpha))$. On the other hand if $(\beta, s(\beta))$ is a Kemer
point of $A$ and $f$ is a Kemer polynomial of $A$ which
corresponds to $(\beta, s(\beta))$, $f$ is not an identity of $A$
and hence not an identity of $A_{j}$ for some $j=1,\dots,s$. It
follows that $(\beta, s(\beta)) \preceq (\alpha_{A_j},
s(\alpha_{A_j}))$. Thus we have two subsets of finite points,
namely $\Kemer(A)$ and $\cup_{1\leq i\leq s} \Kemer(A_{i})$ in a
partially ordered set in $(\mathbb{Z}^{+})^{r}\times
(\mathbb{Z}^{+} \cup \infty)$ such that for any point $(u,s(u))$
in any subset there is a point $(v,s(v))$ in the other subset with
$(u,s(u))\preceq (v,s(v))$. Since $\Kemer(A)$ and $\cup_{1\leq
i\leq t}\Kemer(A_{i})$ are maximal, they must coincide. In
particular, note that the polynomial $f$ above must be a
non-identity (and hence Kemer) of $A_{j}$ for some $j=1,\dots,t$.
It remains to show that a Kemer polynomial of $A_{j}$ for
$j=1,\dots,t$ is a Kemer polynomial of $A$, but this is clear.
\end{proof}

Thus our $T$-ideal $\Gamma$ (the $T$-ideal of identities of a
$G$-graded affine algebra) contains $\id(A)=\id(A_{1}\times
\cdots\times A_{s})$ where $A_i$ are basic algebras.

As noted in Remark \ref{comparison}, $\Ind({\Gamma}) \subseteq
\Ind({A})$ and if a $\alpha$ is a point in $E_{0}(\Gamma)\cap
E_{0}(A)$ (i.e. is extremal for both ideals) then
$s_{\Gamma}(\alpha)\leq s_{A}(\alpha)$.

Our aim now (roughly speaking) is to replace the finite
dimensional algebra $A$ with a finite dimensional algebra $A^{'}$
with $\Gamma \supseteq \id_{G}(A^{'})$ but \textit{PI} ``closer''
to $\Gamma$.

Here is the precise statement (see \cite{BR} for the ungraded
version).

\begin{proposition}\label{same Kemer polynomials}

Let $\Gamma$ and $A$ be as above. Then there exists a $G$-graded
finite dimensional algebra $A^{'}$ with the following properties:

\begin{enumerate}

\item

$\Gamma \supseteq \id_{G}(A^{'})$

\item

The Kemer points of $\Gamma$ coincide with the Kemer points of
$A^{'}$.

\item

Any Kemer polynomial of $A^{'}$ (i.e. a Kemer polynomial which
corresponds to a Kemer point of $A^{'}$) is not in $\Gamma$ (i.e.
$\Gamma$ and $A^{'}$ have the same Kemer polynomials).

\end{enumerate}

\end{proposition}

\begin{remark} The proof is similar but not identical to the proof in \cite{BR}.
For the reader convenience we give a complete proof here.
\end{remark}

\begin{proof}

Let $(\alpha, s(\alpha))$ be a Kemer point of $A$ (i.e. it
corresponds to some of the basic components of $A$). After
renumbering the components we can assume that $(\alpha,
s(\alpha))$ is the Kemer point of $A_{1}$,\ldots, $A_{u}$ and not
of $A_{u+1}$,\ldots, $A_{s}$ . Suppose that $(\alpha, s(\alpha))$
is not a Kemer point of $\Gamma$. Note that since $\Gamma
\supseteq \id_{G}(A)$, there is no Kemer point $(\delta,
s(\delta))$ of $\Gamma$ with $(\delta, s(\delta)) \succeq (\alpha,
s(\alpha))$ and hence any Kemer polynomial of $A$ which
corresponds to the point $(\alpha, s(\alpha))$ is in $\Gamma$. Now
for $i=1,\ldots,u$, let $\mathcal{A}_{i}$ be the relatively free
algebra of $A_i$. For the same indices let $S_{i}$ be the
$T$-ideal generated by all Kemer polynomials of $A_{i}$ and let
$S_{\mathcal{A}_{i}}$ be the ideal of $\mathcal{A}_{i}$ generated
by the evaluations of $S_{i}$ on $\mathcal{A}_{i}$. By Corollary
\ref{representability} we have that ${\mathcal{A}_{i}
}/S_{\mathcal{A}_{i}}$ is representable.

\begin{claim} For any $i=1,\ldots,u$, if $(\beta, s(\beta))$ is any Kemer point of
${\mathcal{A}_{i} }/S_{\mathcal{A}_{i}}$, then $(\beta, s(\beta))
\prec (\alpha_{A_i}, s(\alpha_{A_i}))$.
\end{claim}

(In this claim one may ignore our assumption above that
$(\alpha_{A_i}, s(\alpha_{A_i}))=(\alpha, s(\alpha))$ for
$i=1,\ldots,u$).

Assume the claim is false. This means
that ${\mathcal{A}_{i} }/S_{\mathcal{A}_{i}}$ has a Kemer point
$(\beta, s(\beta))$ for which $(\beta, s(\beta))$ and
$(\alpha_{A_i}, s(\alpha_{A_i}))$ are either not comparable or
$(\beta, s(\beta)) \succeq (\alpha_{A_i}, s(\alpha_{A_i}))$.

If $(\beta, s(\beta))$ and $(\alpha_{A_i}, s(\alpha_{A_i}))$ are
not comparable then there is an element $g$ in $G$ with $\beta_{g}
> \alpha_{A_{i},g}$. But this contradicts $\id_{G}(A_i)\subset
\id_{G}({\mathcal{A}_{i} }/S_{\mathcal{A}_{i}})$ for if $f$ is a
Kemer polynomial for the Kemer point $(\beta, s(\beta))$ of
${\mathcal{A}_{i} }/S_{\mathcal{A}_{i}}$, it must vanish on
$A_{i}$ and hence is in $\id_{G}(A_i)$. The same argument yields a
contradiction in case $(\beta, s(\beta)) \succ (\alpha_{A_i},
s(\alpha_{A_i}))$.

Assume now $(\beta, s(\beta)) = (\alpha_{A_i}, s(\alpha_{A_i}))$
and let $f$ be a Kemer polynomial of the Kemer point $(\beta,
s(\beta))$ of ${\mathcal{A}_{i} }/S_{\mathcal{A}_{i}}$. The
polynomial $f$ is not in $\id_{G}(\mathcal{A}_{i}
/S_{\mathcal{A}_{i}})$ and hence is not in $\id_{G}(A_i)$. Hence
$f$ is a Kemer polynomial of $A_{i}$ and therefore, by
construction, it is in $\id_{G}({\mathcal{A}_{i}
}/S_{\mathcal{A}_{i}})$. This is a contradiction and the claim is
proved.

\bigskip

We replace now the algebras $A_{1}\times \cdots\times A_{u}$ by
${\mathcal{A}_{i} }/S_{\mathcal{A}_{i}}$ (in the product $A=
A_{1}\times \cdots\times A_{s}$). Clearly, the set of Kemer points
of the algebra ${\mathcal{A}_{1} }/S_{\mathcal{A}_{1}} \times
\cdots\times {\mathcal{A}_{u} }/S_{\mathcal{A}_{u}}\times
A_{u+1}\times\cdots\times A_{s}$ is strictly contained in the set
of Kemer points of $A_{1}\times \cdots\times A_{u}\times
A_{u+1}\times\cdots\times A_{s}$ so part $1$ and $2$ of the
proposition will follow by induction if we show that
$\id_{G}({{\mathcal{A}_{1} }/S_{\mathcal{A}_{1}} \times
\cdots\times {\mathcal{A}_{u} }/S_{\mathcal{A}_{u}}\times
A_{u+1}\times\cdots\times A_{s}}) \subseteq \Gamma$. To see this
note that ${\mathcal{A}_{1}}/S_{\mathcal{A}_{1}} \times
\cdots\times {\mathcal{A}_{u} }/S_{\mathcal{A}_{u}} = B/S_{B}$
where $B= \mathcal{A}_{1}\times\cdots\times \mathcal{A}_{u}$,
$S_{B}$ is the ideal of $B$ generated by all evaluations of $S$ on
$B$ and $S$ is the $T$-ideal generated by all polynomial which are
Kemer with respect to $A_{i}$ for some $i=1,\ldots,u$.

Let $z\in \id_{G}({\mathcal{A}_{1} }/S_{\mathcal{A}_{1}} \times
\cdots\times {\mathcal{A}_{u} }/S_{\mathcal{A}_{u}}\times
A_{u+1}\times\cdots\times A_{s})= \langle
\id_{G}({\mathcal{A}_{1}\times\cdots\times \mathcal{A}_{u}}) + S
\rangle_{T} \cap \id_{G}(A_{u+1})\cap\cdots\cap \id_{G}(A_{s})$
and write $z= h + f$ where $h \in
\id_{G}({\mathcal{A}_{1}\times\cdots\times \mathcal{A}_{u}})$ and
$f\in S$. Clearly, we may assume that $f$ is a Kemer polynomial of
$A_{i}$ for some $1\leq i \leq u$. Now since the Kemer point of
$A_{i}$, $i=1,\ldots,u$, is maximal among the Kemer points of $A$,
$f$ and hence $h=z-f$ are in $\id_{G}({A_{u+1}\times\cdots\times
A_{s}})$. It follows that $h \in
\id_{G}({\mathcal{A}_{1}\times\cdots\times \mathcal{A}_{u}}) \cap
\id_{G}({A_{u+1}\times\cdots\times
A_{s}})=\id_{G}(A_{1}\times\cdots\times A_{s}) \subseteq \Gamma$.
But $S \subseteq \Gamma$ (since the point $(\alpha, s(\alpha))$ is
not a Kemer point of $\Gamma$) and hence $z = h+f \in \Gamma$ as
desired.

Now for the proof of (3) we may assume that $\Gamma$ and $A$ have
the same Kemer points. Let $(\alpha,s(\alpha))$ be such a Kemer
point and assume that some Kemer polynomials which correspond to
$(\alpha,s(\alpha))$ are in $\Gamma$. After renumbering the basic
components of $A$ we may assume that $(\alpha,s(\alpha))$ is the
Kemer point of $A_{i}$, $i=1,\ldots,u$. We repeat the argument above
but now instead of taking the set of all Kemer polynomials of the
point $(\alpha,s(\alpha))$ we take only the set of Kemer polynomial
of $(\alpha,s(\alpha))$ which are contained in $\Gamma$. This
completes the proof of the proposition.
\end{proof}

\begin{remark} The proof of Proposition \ref{same Kemer polynomials}
can be sketched as follows: We start with $\Gamma \supseteq
\id_{G}(A)$ where $A$ is finite dimensional and $G$-graded. Let
$\mathcal{A}$ be the relatively free algebra of $A$. By Lemma
\ref{generic} $\mathcal{A}$ is representable and hence we may
consider the trace values of elements of $\mathcal{A}_{e}$. Let
$\Gamma_{0}\subseteq \Gamma$ be the maximal $T$-ideal in $\Gamma$
which is closed under trace multiplication. By Proposition
\ref{product}, $\Gamma_{0}$ \underline{contains the}
$T$-\underline{ideal} \underline{generated by Kemer polynomials
of} $A$ \underline{which are contained in} $\Gamma$. Now, it
follows from \cite{Bei}, that $\mathcal{A}/I_{\Gamma_{0}}$ is
representable, where $I_{\Gamma_{0}}$ is the ideal generated by
all evaluations of $\Gamma_{0}$ on $\mathcal{A}$. Hence
$\mathcal{A}/I_{\Gamma_{0}}$ is \textit{PI}-equivalent to some
finite dimensional algebra $A^{'}$ with $\id_{G}(A^{'})\subseteq
\Gamma$. Finally, one sees that either the Kemer points of $A^{'}$
are smaller comparing to those of $A$ or else the intersection of
$\Gamma$ with Kemer polynomials of $A^{'}$ is zero (it's preimage
in $\mathcal{A}$ must be in $\Gamma_{0}$).
\end{remark}

\begin{corollary}[$\Gamma$-Phoenix property]\label{Phoenix property for Kemer polynomials}

Let $\Gamma$ be a $T$-ideal as above and let $f$ be a Kemer
polynomial of $\Gamma$. Then it satisfies the $\Gamma$-Phoenix
property.

\end{corollary}

\begin{proof}

By Proposition \ref{same Kemer polynomials} $f$ is a Kemer
polynomial of a Kemer point of a finite dimensional algebra $A$.
By Corollary \ref{Kemer polynomials of A are Phoenix} for every
polynomial $f^{'}\in \langle f \rangle$ there is a polynomial
$f^{''}\in \langle f^{'} \rangle$ which is Kemer for $A$. Applying
once again Proposition \ref{same Kemer polynomials} the result
follows.
\end{proof}
\end{section}

\begin{section}{Zubrilin-Razmyslov traces and representable spaces}

As explained in the introduction, the proof of Representability of
$G$-graded affine algebras (Theorem \ref{PI-equivalence-affine})
has two main ingredients. One is the \textit{Phoenix property} of
Kemer polynomials (which is the final statement of the last
section) and the other one (which is our goal in this section) is
the construction of a \textit{representable} algebra which we
denoted there by $B_{(\alpha, s)}$.

Choose a Kemer point $(\alpha, s(\alpha))$ of $\Gamma$ and let
$S_{(\alpha, s(\alpha))}$ be the $T$-ideal generated by all Kemer
polynomial which correspond to the point $(\alpha, s(\alpha))$
with at least $\mu$-folds of small sets. Note that by Remark
\ref{finite number of variables in Kemer polynomials} we may
assume that the total number of variables in these polynomials is
bounded. Let $\mathcal{W}_{\Gamma}$ be the relatively free algebra
of $\Gamma$. In what follows it will be important to assume (as we
may by Corollary \ref{reduction to affine G-grading}) that
$\mathcal{W}_{\Gamma}$ is affine. Since we will not need to refer
explicitly to the variables in the construction of
$\mathcal{W}_{\Gamma}$ we keep the notation $X_{G}$ of $G$-graded
variables for a different purpose. Let $X_{G}= \cup X_{g}$ be a
set of $G$-graded variables where $X_{g}$ has cardinality $\mu
\alpha_{g}+ s(\alpha)(\alpha_{g}+1)$ (i.e. enough $g$-variables to
support Kemer polynomials with $\mu$ small sets and possibly
$s(\alpha)$ big sets which are $g$-homogeneous). Let
$\mathcal{W}^{'}_{\Gamma}=\mathcal{W}_{\Gamma}\ast F\{X_{G} \}$
($G$-graded) and
$\mathcal{U}_{\Gamma}=\mathcal{W}^{'}_{\Gamma}/I_{1}$ where
$I_{1}$ is the ideal generated by all evaluations of $\Gamma$ on
$\mathcal{W}^{'}_{\Gamma}$. Clearly $\id_{G}(\mathcal{U}_{\Gamma})=\Gamma$.

Consider all possible evaluations in $\mathcal{W}_{\Gamma}^{'}$ of
the Kemer polynomials in $S_{(\alpha, s(\alpha))}$ in such a way
that \textit{precisely} $\mu$ folds of small sets and all big sets
(and no other variables) are evaluated on different variables of
$X_{G}$. Denote by $S_{0}$ the space generated by these
evaluations. Note that every nonzero polynomial in $S_{0}$ has an
evaluation of that kind which is nonzero in $\mathcal{U}_{G}$. In
other words $S_{0} \cap I_{1} =0$.

Our aim is to construct a representable algebra $B_{(\alpha,
s(\alpha))}$ and a $G$-graded epimorphism $\varphi:
\mathcal{U}_{G}\longrightarrow B_{(\alpha, s(\alpha))}$ (in particular
$\Gamma \subseteq \id_{G}(B_{(\alpha, s(\alpha))})$, such that
$\varphi$ maps the space $S_{0}$ isomorphically into $B_{(\alpha,
s(\alpha))}$. Let us introduce the following general
terminology.

\begin{definition}

Let $W$ be a $G$-graded algebra over a field $F$. Let $S$ be an $F$-subspace of $W$.
We say that $S$ is a \textit{representable space} of $W$ if there exists a $G$-graded representable
algebra $B$ and a $G$-graded epimorphism
$$\phi:W\longrightarrow B$$

such that $\phi$ maps $S$ isomorphically into $B$.

\end{definition}

We can now state the main result of this section.

\begin{theorem}[Representable Space]\label{Representable Space}
With the above notation, there exists a representable algebra
$B_{(\alpha, s(\alpha))}$ and a $G$-graded surjective homomorphism
$\varphi: \mathcal{U}_{G}\longrightarrow B_{(\alpha, s(\alpha))}$
(hence $\Gamma \subseteq \id_{G}(B_{(\alpha, s(\alpha))})$) and such
that $\varphi$ maps the space $S_{0}$ isomorphically into
$B_{(\alpha, s(\alpha))}$. In particular the space $S_{0}$ is
representable.
\end{theorem}

It is appropriate to view the theorem above as a ``partial
success": Our final goal is to show that the algebra
$\mathcal{U}_{G}$ is representable but here we ``only" prove that
the subspace $S_{0}$ (spanned by Kemer polynomials) is
representable. In order to complete the proof of Theorem
\ref{PI-equivalence-affine} we must invoke the Phoenix property of
Kemer polynomial. The reader may want to ``jump" to section
\ref{Representability of affine $G$-graded
algebras} and see how to
finalize the proof of Theorem \ref{PI-equivalence-affine} using
the Phoenix property of Kemer polynomials and the representability
of the space $S_{0}$.

The construction of $B_{(\alpha, s(\alpha))}$ is based on two key
lemmas. One is the ``Zubrilin-Razmyslov identity" and the second
is a lemma named as the ``interpretation lemma". We start with the
``Zubrilin-Razmyslov identity (see \cite{BR} for the ungraded
case).

Let $\{x_{1,e},\ldots,x_{n,e},x_{n+1,e}\}$ be a set of
\underline{$e$-variables}, $Y_{G}$ a set of arbitrary $G$-graded
variables and $z=z_{e}$ an additional \underline{$e$-variable}.
For a given $G$-graded polynomial
$f(x_{1,e},\ldots,x_{n,e},x_{n+1,e};Y_{G})$, multilinear in the
$x$'s, we define $u_{j}^{z}(f)$ to be the homogeneous component of
degree $j$ in $z$ in the polynomial
$f((z+1)x_{1,e},\ldots,(z+1)x_{n,e},x_{n+1,e};Y_{G})$. In other
words $u_{j}^{z}(f)$ is the sum of all polynomials obtained by replacing
$x_{i,e}$ by $zx_{i,e}$ in $j$ positions from $\{x_{1,e},\ldots,x_{n,e}\}$.
Clearly, if $f$ alternates in the variables
$\{x_{1,e},\ldots,x_{n,e}\}$ then $u_{j}^{z}(f)$ alternates in
these variables as well. Note that for any $1 \leq i,j \leq n$,
the operators $u_{i}^{z}$ and $u_{i}^{z}$ commute.

Let $A$ be any $G$-graded algebra over $F$. Let $f$ be as above
and assume it alternates in $\{x_{1,e},\ldots,x_{n,e}\}$. Consider
the polynomial

$$\widetilde{f}(x_{1,e},\ldots,x_{n,e},x_{n+1,e};Y_{G})=
f(x_{1,e},\ldots,x_{n,e},x_{n+1,e};Y_{G})$$ $$- \sum_{k=1}^{n}
f(x_{1,e},\ldots,
x_{k-1,e},x_{n+1,e},x_{k+1,e},\dots,x_{n,e},x_{k,e};Y_{G}).$$

Note that $\widetilde{f}(x_{1,e},\ldots,x_{n,e},x_{n+1,e};Y_{G})$
alternates in the variables $\{x_{1,e},\ldots,x_{n+1,e}\}$. The
proof of the following proposition is identical to the proof of
Proposition 2.44 in \cite{BR} and hence is omitted.

\begin{proposition}[Zubrilin-Razmyslov identity]\label{Zubrilin-Razmyslov-Lemma}

With the above notation: if
$$\widetilde{f}(x_{1,e},\ldots,x_{n,e},x_{n+1,e};Y_{G})$$ is a
$G$-graded identity of $A$ then also is
$$\sum_{j=0}^{n}(-1)^{j}u_{j}^{z}(f(x_{1,e},\ldots,x_{n,e}, z^{n-j}x_{n+1,e};Y_{G}).$$

\end{proposition}

\begin{lemma}[Interpretation Lemma]\label{Interpretation Lemma}

Let $A$ be a $G$-graded algebra over a field $F$ and $I$ a
$G$-graded ideal of $A$. Let $\Lambda
=F[\theta_{1},\dots,\theta_{n}]$ be a commutative, finitely
generated $F$-algebra. Suppose $\Lambda$ acts on $I$ as linear
operators and the action commutes with the multiplication in $A$
(we view the elements of $\Lambda$ as homogeneous of degree $e\in
G$). Consider the extension of $A$ by commuting $e$-variables
$\{\lambda_{1},\dots,\lambda_{n}\}$ and let $K$ be the $G$-graded
ideal of $A[\lambda_{1},\dots,\lambda_{n}]$ generated by the
elements $(\lambda_{i}x-\theta_{i}x)$, $i=1,\dots,n$ and $x\in I$.
Then the natural map $A \rightarrow A^{'}=
A[\lambda_{1},\dots,\lambda_{n}]/K$ is an embedding.

\end{lemma}

\begin{proof}
We prove the lemma by giving an explicit description of $A^{'}$.
Let $V$ be a complement of $I$ in $A$ (as an $F$-vector space).
Since $I$ is $G$-graded we may assume that $V$ is spanned by
homogeneous elements. Let
$F[\lambda_{1},\dots,\lambda_{n}]\otimes_{F}V$ be the extension
$V$ by the $\lambda_{i}$'s and consider the subspace
$C=F[\lambda_{1},\dots,\lambda_{n}]\otimes_{F}V + I$ of
$A[\lambda_{1},\dots,\lambda_{n}]$. We introduce an action of
$F[\lambda_{1},\dots,\lambda_{n}]$ on $C$ as follows: The action
on $F[\lambda_{1},\dots,\lambda_{n}]\otimes_{F}V$ is the obvious
one where the action on $I$ is given by $\lambda_{i}x=\theta_{i}x$
for every $x\in I$. Next we introduce a multiplication on $C$: Let
$\mu: F[\lambda_{1},\dots,\lambda_{n}] \rightarrow
F[\theta_{1},\dots,\theta_{n}]$ be the algebra map defined by
$\lambda_{i} \mapsto \theta_{i}$. Take $v_{1}$ and $v_{2}$ in $V$
and let their multiplication in $A$ be given by $v_{1}v_{2}=v_{3}
+ a$ where $v_{3}\in V$ and $a\in I$. Then we define
$(r_{1}\otimes v_{1})(r_{2}\otimes v_{2})=r_{1}r_{2}\otimes
v_{3}+\mu(r_{1}r_{2})a$. The product of $r\otimes v$ and an
element of $I$ is defined in the same way (using the map $\mu$).
Now it is clear that the algebras $C$ and $A^{'}$ are isomorphic
and that $A$ is embedded in $C$.
\end{proof}

\begin{remark}
A similar statement can be proved for algebras over an arbitrary commutative,
noetherian ring $R$. Instead of the space $V$ one can consider the coset
representatives of $I$ in $A$ and be more ``careful" with the addition operation.
\end{remark}

We can turn now to the construction of $B_{(\alpha, s(\alpha))}$.

Consider the ideal $I_{2}$ of $\mathcal{W}^{'}_{\Gamma}$ generated
by all elements of the form $x_{g}zx_{g}$, $z\in
\mathcal{W}^{'}_{\Gamma}$ and $x_{g}\in X_{g}$. Clearly the
natural map $\mathcal{W}^{'}_{\Gamma}\longrightarrow
\mathcal{W}^{'}_{\Gamma}/I_{2}$ maps the space $S_{0}$
isomorphically. To simplify the notation we denote the image of
$S_{0}$ in $\mathcal{W}^{'}_{\Gamma}/I_{2}$ again by $S_{0}$. Note
that the ideal of $\mathcal{W}^{'}_{\Gamma}/I_{2}$ generated by
the elements of $X_{G}$ is nilpotent.

In order to construct the algebra $B_{(\alpha, s(\alpha))}$ we
construct a sequence of algebras $B^{(r)}$, $r=0,\ldots,t$, where
$B^{(0)}= \mathcal{W}^{'}_{\Gamma}/I_{2}$, $B^{(t)}= B_{(\alpha,
s(\alpha))}$, and $B^{(r+1)}$ is obtained from $B^{(r)}$ by first
extending its centroid with a certain finite set of indeterminates
$\lambda_{i,1},\ldots,\lambda_{i,n}$ ($n$ is the cardinality of an
$e$-small set) and then by moding out from
$B^{(r)}[\lambda_{i,1},\ldots,\lambda_{i,n}]$  a suitable ideal
which we denote by $J_{a_r}$. Our main tasks will be

\begin{enumerate}

\item

to show that $B_{(\alpha, s(\alpha))}$ is a finite module over its
centroid (and hence representable by \cite{Bei}).

\item

to show that the subspace of $B^{(r)}$ spanned by the image of
$S_{0}$ is mapped isomorphically into $B^{(r+1)}$.
\end{enumerate}

We choose an essential Shirshov base $\{a_0,\ldots,a_{t-1}\}$ of
$\mathcal{W}^{'}_{\Gamma}$. As shown in Proposition \ref{Essential
Shirshov base in the e-component}, these elements can be taken
from $(\mathcal{W}^{'}_{\Gamma})_{e}$. Moreover, since the ideal
generated $X_{G}$ is nilpotent we can assume the $a_i$'s are
$X_{G}$-free. Clearly, the (images of) elements
$\{a_0,\ldots,a_{t-1}\}$ form an essential Shirshov base of
$B^{(0)}=\mathcal{W}^{'}_{\Gamma}/I_{2}$. Moreover since the
construction of $B^{(j+1)}$ consists of extending the centroid of
$B^{(j)}$ and moding out by a certain ideal, the image of
$\{a_0,\ldots,a_{t-1}\}$ in $B^{(j)}$ is an essential Shirshov
base for $B^{(j)}$, $j=0,\ldots,t$.  We are now ready to define
$B^{(j+1)}$.

Let $B^{(j+1)}=B^{(j)}[\lambda_{j,1},\ldots,\lambda_{j,n}]/J_{a_j}$
where $J_{a_j}$ is the ideal generated by the expression
$$a_{j}(a_{j}^n+\lambda_{j,1} a_{j}^{n-1}+\lambda_{j,2}
a_{j}^{n-2}+\cdots +\lambda_{j,n})=a_{j}^{n+1}+\lambda_{j,1}
a_{j}^{n}+\lambda_{j,2} a_{j}^{n-1}+\cdots +\lambda_{j,n}a_{j}.$$

From the definition of $J_{a_j}$ it follows that the image of
$a_{j}$ in $B^{(t)}$ is integral over the centroid. In other words
$B^{(t)}$ has an essential Shirshov base consisting of integral
elements and so it is a finite module over its centroid. This
proves (1).

For the proof of (2) we need to show
that $\overline{S}_{0}$, the image of $S_{0}$ in
$B^{(r)}[\lambda_{r,1},\ldots,\lambda_{r,n}]$, intersects trivially
$J_{a_{r}}$. This is an immediate consequence of the lemma below.
We insist in rephrasing it as a separate lemma in order to emphasize that
its proof is independent of the inductive process presented above.

\bigskip

Let $W$ be a \textit{PI}, $G$-graded affine algebra over a field
$K$. Let $(\alpha, s(\alpha))$ be a Kemer point of $W$. Fix a
configuration of big sets according to the
Kemer point $(\alpha, s(\alpha))$, that is
we fix an $s(\alpha)$-tuple $u=(g_{1},\ldots,g_{s(\alpha)})$ in $G^{s(\alpha)}$.
Let $X_{(\alpha, s(\alpha),u)}$ be a set of graded variables
with $\mu$ small sets of $g$-variables of cardinality $\alpha_{g}$,
for every $g\in G$ and big sets $\Lambda_{g_i}$ of $g_{i}$-variables,
of cardinality $\alpha_{g_i}+1$ for $i=1,\ldots,s(\alpha)$.
Thus the total number of variables in $X_{(\alpha, s(\alpha),u)}$ is given by

$$\mu \sum_{g\in G} \alpha_{g} + \sum_{i=1}^{s(\alpha)}(\alpha_{g_i}+1).$$

(Note that here we don't require that $W$ has a Kemer polynomial with such configuration.)

Consider the algebra $\widehat{W}=W\ast\{X_{(\alpha,
s(\alpha),u)}\}/(I_{1}+I_{2})$ where $I_{1}$ is the ideal of
$W\ast \{X_{(\alpha, s(\alpha),u)}\}$ generated by all evaluations
of $\id_{G}(W)$ on $W\ast\{X_{(\alpha, s(\alpha),u)}\}$ and
$I_{2}$ the ideal of $W\ast\{X_{(\alpha, s(\alpha),u)}\}$
generated by elements of the form $x_{g}w^{'}x_{g}$ where $w^{'}
\in W\ast\{X_{(\alpha, s(\alpha),u)}\}$. Consider the scalar
extension $\widehat{W}[\lambda_{1},\dots,\lambda_{n}]$ of
$\widehat{W}$ where $\lambda_{i}$ are indeterminates and
$n=\alpha_{e}$. Given an element $b\in W$, let $J_{b}$ be the
ideal of $W\ast\{X_{(\alpha, s(\alpha))}\}/(I_{1}+I_{2})$
generated by the the expression
$$
b(b^n+\lambda_{1}
b^{n-1}+\lambda_{2}b^{n-2}+\cdots
+\lambda_{n})=b^{n+1}+\lambda_{1} b^{n}+\lambda_{2} b^{n-1}+\cdots
+\lambda_{n}b.
$$

\begin{lemma}
Let $S$ be the subspace of $\widehat{W}$ spanned by all polynomials
in the graded variables of $X_{(\alpha, s(\alpha),u)}$ which alternate
on small and big sets according to the configuration described above.
Then the restriction to $S$ of the natural map

$$\widehat{W}\rightarrow \widehat{W}[\lambda_{1},\dots,\lambda_{n}]/J_{b}$$

is an embedding.
\end{lemma}

\begin{proof}

We need to show that if $f\in S \cap J_{b}$ then
$f=0$. Being in $J_{b}$, $f$ has the form
$$\sum p_{i}(X, \lambda)(b^{n+1}+\lambda_{1} b^{n}+\lambda_{2}
b^{n-1}+\cdots +\lambda_{n}b)q_{i}(X, \lambda)$$

for some $p_i$ and $q_i$. Furthermore $f$ can be written as sums
of expressions of the form

$$p_{1}(X_1) (b^{n+1}+\lambda_{1}
b^{n}+\lambda_{2} b^{n-1}+\cdots +\lambda_{n}b)p_{2}(X_2)
g(\lambda)$$

where

\begin{enumerate}

\item

$p_{i}(X_i)$ are polynomials in variables of $X_{(\alpha, s(\alpha),u)}$.

\item

from the definition of the ideal $I_{2}$ above we can assume that
all variables of $X_{(\alpha, s(\alpha),u)}$ appear exactly once
in either $p_{1}(X_1)$ or $p_{2}(X_2)$.

\item

the polynomials $p_{i}(X_i)$ are free of $\lambda$'s.
\item

$g(\lambda)$ is $X_{(\alpha, s(\alpha),u)}$ free.
\end{enumerate}

Let us alternate the variables of $X_{(\alpha, s(\alpha),u)}$
(according to its decomposition to small and big sets).

\begin{note} Since the polynomial $f$ is already alternating
in the variables of $X_{(\alpha, s(\alpha),u)}$, the alternation
above as the effect of multiplying the polynomial $f$ by an
integer $\pi$ which is a product of factorials. Since the
characteristic of the field $F$ is zero we have $\pi \neq 0$. This
is why the result of this section is not characteristic free. In
positive characteristics (proceeding as above) one can conclude
only that the \textit{images} of the alternating operators form a
representable space.
\end{note}

Applying this alternation on each summand
$$h=p_{1}(X_1) (b^{n+1}+\lambda_{1} b^{n}+\lambda_{2}
b^{n-1}+\cdots +\lambda_{n}b)p_{2}(X_2) g(\lambda) \in J_{b}$$

yields a polynomial
$$\widehat{h}= \sum \sgn(\sigma)p_{1}(X_{\sigma,1})
(b^{n+1}+\lambda_{1} b^{n}+\lambda_{2} b^{n-1}+\cdots
+\lambda_{n}b)p_{2}(X_{\sigma,2}) g(\lambda)$$

that alternates on small sets and big sets of $X_{(\alpha,
s(\alpha),u)}$. We will present an interpretation of the variables
$\lambda_{i}$ which annihilates $\widehat{h}=h(x_1,\dots,x_n, y)$.
But then, since $f$ is free of $\lambda$'s (that is the
interpretation does not annihilate $f$) the result follows.

Recall the operators $u^{b}_{j}$ from the Zubrilin-Razmyslov
identity (Proposition \ref{Zubrilin-Razmyslov-Lemma}). Factoring
the algebra $\widehat{W}[\lambda_{1},\ldots,\lambda_{n}]$ by the
ideal generated by $\lambda_{j}-u^{b}_{j}$ applied to the
polynomials of $S$ yields the algebra
$$D=\widehat{W}[\lambda_{1},\ldots,\lambda_{n}]/ <(\lambda_{j}-u^{b}_{j})S_{0}>.$$

Invoking the Interpretation Lemma (Lemma \ref{Interpretation
Lemma}) for $\theta_{j}=u^{b}_{j}$, $j=1,\dots,n$ we have that
$\widehat{W}$ and in particular $S$, are embedded in
$D$ and hence the interpretation does not annihilate $f$. Finally,
let us see that the substitution $\theta_{j}=u^{b}_{j}$,
$j=1,\dots,n$ annihilates

$\widehat{h}= \sum \sgn(\sigma)p_{1}(X_{\sigma,1})
(b^{n+1}+\lambda_{1} b^{n}+\lambda_{2} b^{n-1}+\cdots
+\lambda_{n}b)p_{2}(X_{\sigma,2}) g(\lambda)$.

Indeed, this follows from Proposition
\ref{Zubrilin-Razmyslov-Lemma} and the fact that $\widehat{h}$ is
alternating on small an big sets which correspond to the Kemer
point $(\alpha,s(\alpha))$. This completes the proof of Theorem
\ref{Representable Space}.
\end{proof}
We close the section with the following general statement. The proof
is similar to the proof \ref{Representable Space} and hence is
omitted. Let $W$, $\overline{W}$, $S$ as in the previous lemma.
\begin{theorem}
The subspace $S$ of $\widehat{W}$ is representable.
\end{theorem}

\end{section}

\begin{section}{Representability of affine $G$-graded algebras} \label{Representability of affine $G$-graded
algebras}

In this section we complete the proof of Theorem
\ref{PI-equivalence-affine}.

Proof of Theorem \ref{PI-equivalence-affine}: Choose a Kemer point
$(\alpha, s(\alpha))$ of $\Gamma$ and let $S_{(\alpha,
s(\alpha))}$ be the $T$-ideal generated by all Kemer polynomial
which correspond to the point $(\alpha, s(\alpha))$ with at least
$\mu$-folds of small sets. Consider the $T$-ideal $\Gamma^{'}=
\langle \Gamma + S_{(\alpha, s(\alpha))} \rangle$. Observe that
the Kemer set of $\Gamma^{'}$ is strictly contained in the Kemer
set of $\Gamma$ (since $(\alpha, s(\alpha))$ is not a Kemer point
of $\Gamma^{'}$). Hence, applying induction (if $(\alpha,
s(\alpha))=0$ is the only Kemer point of $\Gamma$ then
$\Gamma=\id_{G}(0)$), there exists a finite dimensional algebra
$A^{'}$ with $\Gamma^{'}=\id_{G}(A^{'})$. We show that $\Gamma$ is
\textit{PI}-equivalent to the algebra $A^{'}\oplus B_{(\alpha,
s(\alpha))}$. Clearly, $\Gamma$ is contained in the intersection
of the $T$-ideals $\id_{G}(A^{'})$ and $\id_{G}(B_{(\alpha,
s(\alpha))})$. For the converse take an identity $f$ of $A^{'}$
which is not in $\Gamma$. We can assume that $f$ is generated by
Kemer polynomials and hence by Corollary \ref{Phoenix property for
Kemer polynomials} it has a corollary $f^{'}$ which is Kemer. But
then it has an evaluation in $\mathcal{W}^{'}_{\Gamma}$ which
yields a nonzero element in $S_{0}$. Applying Theorem
\ref{Representable Space} we have that $S_{0}\cap
\id_{G}(B_{(\alpha, s(\alpha))})=0$ and the result follows. This
completes the proof of Theorem \ref{PI-equivalence-affine}.

\begin{corollary}[$G$-graded representability-affine]\label{representability of G-graded-affine}
The relatively free $G$-graded algebra $\Omega_{F,G}/\id_{G}(W)$ is
representable. That is, $\Omega_{F,G}/\id_{G}(W)$ can be embedded in a
finite dimensional algebra over a (sufficiently large) field $K$.

\end{corollary}

\begin{proof} By Theorem \ref{PI-equivalence-affine} we know
that there exists a finite dimensional $G$-graded algebra $A$ with
$\id_{G}(W)=\id_{G}(A)$. Consequently the corresponding relatively
free algebras $\mathcal{W}$ and $\mathcal{A}$ are isomorphic.
Since $\mathcal{A}$ is representable the result follows.
\end{proof}

We close the section with a theorem which is a corollary of Theorem
\ref{PI-equivalence-affine}, the reduction to
direct products of basic algebras (Proposition \ref{basic algebras})
and Kemer's Lemma $2$ (Lemma \ref{Kemer's Lemma 2}).

\begin{theorem} \label{TheoremAffineAdequateModel}
Every variety $M_{W}$ of an affine algebra $W$ can be generated by a
finite dimensional algebra which is a finite direct product of basic
algebras $B_1,\dots,B_n$.
\end{theorem}

Note that we can view the basic algebras $B_i$ as \textit{adequate
models} of the variety: this means that combinatorial parameters,
namely, cardinalities of small sets and number of big sets of
Kemer polynomials coincide with dimensions of graded components of
the semisimple part of $B_i$ and the nilpotency index of $J(B_i)$.
\end{section}

\begin{section}{Specht Problem for $G$-graded affine
algebras}\label{Specht Problem for $G$-graded affine algebras}

In this section we prove Theorem \ref{Specht}.

Let $W$ be an affine \textit{PI} $G$-graded algebra over $F$ and let
$\id_{G}(W)$ be its $T$-ideal of $G$-graded identities. Our goal is
to find a finite generating set for $\id_{G}(W)$. Since we are
assuming that $W$ is (ungraded) \textit{PI}, we have by Theorem
\ref{PI-equivalence-affine} that $\id_{G}(W)=\id_{G}(A)$ where $A$
is an algebra over $K$ (a field extension of $F$), $G$-graded and
finite dimensional. If the dimension of $A$ is $m$ say, then clearly
$W$ satisfies $c_{m+1}$, the ungraded Capelli identity on $2(m+1)$
variables, or equivalently, the \textit{finite} set of $G$-graded
identities $c_{(G,m+1)}$ which follow from $c_{m+1}$ by designating
$G$-degrees to its variables.

Now, observe that any $T$-ideal of $G$-graded identities is
generated by at most a countable number of graded identities
(indeed, for each $n$ the space of multilinear $G$-graded identities
of degree $n$ is finite dimensional) hence we may take a sequence of
graded identities $f_{1},\ldots,f_{n},\ldots$ which generate
$\id_{G}(W)$. Clearly, since the set $c_{(G,m+1)}$ is finite, in order
to prove the finite generation of $\id_{G}(W)$ it is sufficient to
show that the ascending chain of graded $T$-ideals $\Gamma_{1}
\subseteq \Gamma_{2} \subseteq \ldots \subseteq \Gamma_{n} \subseteq
\ldots$, where $\Gamma_{n}$ is the $T$-ideal generated by the
polynomials $c_{(G,m+1)} \cup \{f_{1},\ldots,f_{n}\}$, stabilizes.

Now by subsection \ref{Finite generation of the relatively free
algebra}, for each $n$, the $T$-deal $\Gamma_{n}$ corresponds to
an affine algebra and hence invoking Theorem
\ref{PI-equivalence-affine} we may replace each $\Gamma_{n}$ by
$\id_{G}(A_{n})$ where $A_{n}$ is a $G$-graded finite dimensional
algebra over a suitable field extension $K_{n}$ of $F$. Clearly,
extending the coefficients to a sufficiently large field $K$ we
may assume all algebras $A_{n}$ are finite dimensional over an
algebraically closed field $K$.

So we need to show that the sequence $\id_{G}(A_{1}) \subseteq
\id_{G}(A_{2}) \subseteq \ldots$ stabilizes in $\Omega_{F,G}$ or
equivalently, that the sequence stabilizes in $\Omega_{K,G}$.
Consider the Kemer sets of the algebras $\{A_{n}\}$, $n \geq 1$.
Since the sequence of ideals is increasing, the corresponding
Kemer sets are monotonically decreasing (recall that this means
that for any Kemer point $(\alpha,s)$ of $A_{i+1}$ there is a
Kemer point $(\alpha^{'}, s^{'})$ of $A_{i}$ with $(\alpha,s)
\preceq (\alpha^{'}, s^{'})$). Furthermore, since these sets are
finite, there is a subsequence $\{A_{i_{j}}\}$ whose Kemer points
(denoted by $E$) coincide. Clearly it is sufficient to show that
the subsequence $\{\id_{G}(A_{i_{j}})\}$ stabilizes and so, in
order to simplify notation, we replace our original sequence
$\{\id_{G}(A_{i})\}$ by the subsequence.

Choose a Kemer point $(\alpha, s)$ in $E$. Clearly we may replace
the algebra $A_{i}$ by a direct product of basic algebras
$A^{'}_{i,1} \times A^{'}_{i,2}\times\cdots\times
A^{'}_{i,u_{i}}\times \widehat{A}_{i,1}\times\cdots\times
\widehat{A}_{i,r_{i}}$ where the $A^{'}_{i,j}$'s correspond to the
Kemer point $(\alpha, s)$ and the $\widehat{A}_{i,l}$ have Kemer
index $\neq (\alpha, s)$ (note that their index may or may not be
in $E$).

Our goal is to replace (for a subsequence of indices $i_{k}$) the
direct product $A^{'}_{i,1} \times A^{'}_{i,2}\times\cdots\times
A^{'}_{i,u_{i}}$ (the basic algebras that correspond to the Kemer
point $(\alpha, s)$) by a certain $G$-graded algebra $B$ such that

$\id_{G}(B \times \widehat{A}_{i,1}\times\cdots\times
\widehat{A}_{i,r_{i}}) = \id_{G}(A_{i})$ for all $i$.

Let us show how to complete the proof assuming such $B$ exists. Replace
the sequence of indices $\{i\}$ by the subsequence $\{i_k\}$.
(Clearly, it is sufficient to show that the subsequence of $T$-ideals
$\{\id_{G}(A_{i_{k}})\}$ stabilizes).

Let $I$ be the $T$-ideal generated by Kemer polynomials of $B$
which correspond to the Kemer point $(\alpha,s)$. Note that the
polynomials in $I$ are identities of the basic algebras
$\widehat{A}_{i,l}$'s. It follows that the Kemer sets of the
$T$-ideals $\{(\id_{G}(A_{i}) + I)\}$ do not contain the point
$(\alpha, s)$ and hence are strictly smaller. By induction we
obtain that the sequence of $T$-ideals

$(\id_{G}(A_{1}) + I) \subseteq (\id_{G}(A_{2}) + I)\subseteq
\ldots $ stabilizes.

On the other hand we claim that $I \cap (\id_{G}(A_{i}) = I \cap
(\id_{G}(A_{j})$ for any $i,j$. This follows at once since $A_{i}
= B \times \widehat{A}_{i,1}\times\cdots\times
\widehat{A}_{i,r_{i}}$ and $I \subseteq
\widehat{A}_{i,1}\times\cdots\times \widehat{A}_{i,r_{i}}$.

Combining the last statements the result follows.

\smallskip

Let us show now the existence of the algebra $B$.

Let $A$ be a $G$-graded basic algebra which corresponds to the Kemer
point $(\alpha, s)$. Let $A=\overline{A} \oplus J(A)$ be the
decomposition of $A$ into the semisimple and radical components. As
shown in Section \ref{Section: Kemer's Lemma $1$},
$\alpha_{g}=\dim(\overline{A}_{g})$ for every $g\in G$ and so, in
particular, the dimension of $\overline{A}$ is determined by
$\alpha$. The following claim is key (see \cite{BSZ}).

\begin{proposition}

The number of isomorphism classes of $G$-graded \textit{
semisimple} algebras of a given dimension is finite.

\end{proposition}

Clearly it is sufficient to show that the number of isomorphisms
classes of $G$-graded semisimple algebras $C$ of a given
dimension, which are $G$-simple, is finite. To see this recall
that the $G$-graded structure is given by a subgroup $H$ of $G$, a
$2$nd cohomology class in $H^{2}(H,K^{*})$ and a $k$-tuple
$(g_{1},\ldots,g_{k})$ in $G^{k}$ where $k^{2} \leq
\dim(\overline{C})$. Clearly the number of subgroups $H$ of $G$
and the number of $k$-tuples are both finite. For the cardinality
of $H^{2}(H,K^{*})$, note that since $K$ is algebraically closed
the cohomology group $H^{2}(H,K^{*})$ coincides with the Schur
multiplier of $H$ which is known to be finite. This proves the
Proposition.

We obtain:

\begin{corollary}

The number of $G$-graded structures on the semisimple components of
all basic algebras which correspond to the Kemer point $(\alpha,s)$
is finite.

\end{corollary}

It follows that by passing to a subsequence $\{i_{s}\}$ we may
assume that all basic algebras that appear in the decompositions
above and correspond to the Kemer point $(\alpha,s)$ have
$G$-graded isomorphic semisimple components (which we denote by
$C$) and have the same nilpotency index ($=s$).

Consider the $G$-graded algebras
$$\widehat{C}_{i}= C \ast K\langle X_{G}\rangle / (I_{i} +
J)$$

where $X_{G}$ is a set of $G$-graded variables of cardinality
$(s-1)\cdot \ord(G)$ (that is $s-1$ variables for each $g\in G$),
$I_{i}$ is the ideal generated by all evaluations of
$\id_{G}(A_{i})$ on $C \ast K\langle X_{G}\rangle$ and $J$ is the
ideal generated by all words in $C \ast K\langle X_{G}\rangle$
with $s$ variables from $X_{G}$.

\begin{proposition}

\begin{enumerate}

\item

The ideal generated by variables from $X_{G}$ is nilpotent.

\item

For any $i$, the algebra $\widehat{C}_{i}$ is finite dimensional.

\item For any $i$, $\id_{G}(\widehat{C}_{i} \times
\widehat{A}_{i,1}\times\cdots\times
\widehat{A}_{i,r_{i}})=\id_{G}(A_{i})$.
\end{enumerate}

\end{proposition}

\begin{proof}

The first part is clear. In order to prove (2) consider a typical
non-zero monomial which represents an elements in the algebra
$\widehat{C}_{i}$. It has the form

$$a_{t_{1}}x_{t_{1}}a_{t_{2}}x_{t_{2}}\cdots
a_{t_{r}}x_{t_{r}}a_{t_{(r+1)}}$$

Since the set $X_{G}$ is finite and also the number of variables
appearing in a non-zero monomial is bounded by $s-1$, we have that
the number of different configurations of these monomials (namely,
the number of different tuples $x_{t_{1}},\ldots,x_{t_{r}}$) is
finite. In between these variables we have the elements
$a_{t_{j}}$, $j=1,\dots,r+1$, which are taken from the finite
dimensional algebra $C$. This proves the second part of the
Proposition. We now show the 3rd part of the Proposition.

Clearly, $\id_{G}(\widehat{A}_{i,j}) \supseteq \id_{G}(A_{i})$ for
$j=1,\dots,r_{i}$. Also, from the definition of $\widehat{C}_{i}$
we have that $\id_{G}(\widehat{C}_{i}) \supseteq \id_{G}(A_{i})$
and so $\id_{G}(\widehat{C}_{i} \times
\widehat{A}_{i,1}\times\cdots\times \widehat{A}_{i,r_{i}})
\supseteq \id_{G}(A_{i})$. For the converse we show that
$\id_{G}(\widehat{C}_{i}) \subseteq \id_{G}(A^{'}_{i,j})$ for
every $j=1,\ldots,u_{i}$. To see this let us take a multilinear,
$G$-graded polynomial
$p=p(x_{i_1,g_{i_1}},\ldots,x_{i_t,g_{i_t}})$ which is a graded
non-identity of $A^{'}_{i,j}$ and show that $p$ is in fact a
graded non-identity of $\widehat{C}_{i}$. Fix a non vanishing
evaluation of $p$ on $A^{'}_{i,j}$ where
$x_{j_1,g_{j_1}}=z_{1},\dots,x_{j_k,g_{j_k}}=z_{k}$ ($k \leq s-1$)
are the variables with the corresponding radical evaluations and
$x_{q_1,g_{q_1}}=c_{1},\dots,x_{q_k,g_{q_k}}=c_{k}$ are the other
variables with their semisimple evaluations. Consider the
$G$-graded map
$$ \eta: C \ast K\langle X_{G}\rangle \longrightarrow A^{'}_{i,j} $$
where

\begin{enumerate}

\item

$C$ is mapped isomorphically.

\item

A subset of $k$ variables $\{y_1,\dots,y_k\}$ of $X_{G}$ (with
appropriate $G$-grading) are mapped onto the set
$\{z_1,\dots,z_k\}$. The other variables from $X_{G}$ are mapped
to zero.
\end{enumerate}

Note that $\eta$ vanishes on $(I_{i} + J)$ and hence we obtain a
$G$-graded map $\overline{\eta}: \widehat{C}_{i} \longrightarrow
A^{'}_{i,j}$. Clearly, the evaluation of the polynomial
$p(x_{i_1,g_{i_1}},\ldots,x_{i_t,g_{i_t}})$ on $\widehat{C}_{i}$
where $x_{q_1,g_{q_1}}=c_{1},\dots,x_{q_k,g_{q_k}}=c_{k}$ and
$x_{j_1,g_{j_1}}=y_{1},\dots,x_{j_k,g_{j_k}}=y_{k}$ is non-zero
and the result follows.
\end{proof}
At this point we have a sequence of $T$-ideals
\begin{eqnarray*}
\id_{G}(\widehat{C}_{1} \times \widehat{A}_{i,1}\times\cdots\times
\widehat{A}_{1,r_{1}}) \subseteq \ldots \subseteq
\id_{G}(\widehat{C}_{i} \times \widehat{A}_{i,1}\times\cdots\times
\widehat{A}_{i,r_{i}})\subseteq \\
\subseteq \id_{G}(\widehat{C}_{i+1} \times
\widehat{A}_{i+1,1}\times\cdots\times \widehat{A}_{i+1,r_{i+1}})
\subseteq \ldots
\end{eqnarray*}

In order to complete the construction of the algebra $B$ (and
hence the proof of the Specht problem) we will show that in fact,
by passing to a subsequence, all $\widehat{C}_{i}$ are $G$-graded
isomorphic. Indeed, since $\id_{G}(A_i) \subseteq
\id_{G}(A_{i+1})$ we have a (surjective) map
$$\widehat{C}_{i}= C \ast K\langle X_{G}\rangle /
(I_{i} + J) \longrightarrow \widehat{C}_{i+1}= C \ast K\langle
X_{G}\rangle / (I_{i+1} + J)$$ Since the algebras
$\widehat{C}_{i}$'s are finite dimensional the result follows.
\end{section}

\begin{section}{Non-affine algebras}\label{Section:Non-affine}

In this section we prove Theorem \ref{PI-equivalence-general} and
Theorem \ref{Specht-general}.

\begin{proof}(Th. \ref{PI-equivalence-general})
We proceed as in \cite{BB} where the Hopf algebra $H$ is replaced by
$(FG)^{*}$, the dual Hopf algebra of the group algebra $FG$. Let $W$
be a \textit{PI} $G$-graded (possibly) non-affine algebra. We
consider the algebra $W^{*}=W\otimes E$ where $E$ is the Grassmann
algebra. Note that the algebra $W^{*}$  is $\mathbb{Z}/2\mathbb{Z}
\times G$-graded where the $G$-grading comes from the $G$-grading on
$W$ and the $\mathbb{Z}/2\mathbb{Z}$-grading comes from the
$\mathbb{Z}/2\mathbb{Z}$-grading on $E$.

By (\cite{BB}, Lemma 1) there exists an affine
$\mathbb{Z}/2\mathbb{Z} \times G$-graded algebra $W_{affine}$ such
that $\id_{\mathbb{Z}/2\mathbb{Z} \times
G}(W^{*})=\id_{\mathbb{Z}/2\mathbb{Z} \times G}(W_{affine})$ and
hence by Theorem \ref{PI-equivalence-affine}, it coincides with
$\id_{\mathbb{Z}/2\mathbb{Z} \times G}(A)$ where $A$ is a finite
dimensional $\mathbb{Z}/2\mathbb{Z} \times G$-graded algebra.
Applying the $*$ operator to $W^{*}$ (and using the fact that
$\id_{G}(W)=\id_{G}(W^{**})$) we obtain that
$\id_{G}(W)=\id_{G}(A^{*})$ as desired.
\end{proof}

\begin{proof}(Th. \ref{Specht-general})
Let $\Gamma$ be the $T$-ideal of $G$-graded identities of $W$. Let
$\Gamma_{1} \subseteq \Gamma_{2} \subseteq \dots.$ be an ascending
sequence of $T$-ideals whose union is $\Gamma$. Since $W$ is
assumed to be \textit{PI} (as in the affine case) we can add to
all $\Gamma_{i}$'s a finite set of $G$-graded identities so that
the ideals obtained correspond to $T$-ideals of \textit{PI}
$G$-graded algebras. By Theorem \ref{PI-equivalence-general}
these $T$-ideals correspond to Grassmann envelopes of finite
dimensional ${\mathbb{Z}/2\mathbb{Z} \times G}$-graded algebras $A_{i}$,
that is we obtain an ascending chain of the form
$\id_{\mathbb{Z}/2\mathbb{Z} \times G}((A_{1})^{*})\subseteq
\id_{\mathbb{Z}/2\mathbb{Z} \times G}((A_{2})^{*}) \subseteq \dots$.
Applying the $*$ operator, we get an ascending chain of $T$-ideals
of identities of finite dimensional algebras so it must stabilize.
The result now follows from the fact that $*$ is an involution.
\end{proof}
We conclude the section with the theorem corresponding to the
theorem \ref{TheoremAffineAdequateModel}.

\begin{theorem}\label{TheoremAdequateModel}
Every variety $M$ of $G$-graded algebras can be generated by a
finite direct product $B^{*}_1 \times \cdots \times B^{*}_r\times
B'_1\times\cdots\times B'_s$ of Grassmann envelopes $B^{*}_i$ of
basic algebras $B_i$ and basic algebras $B'_j$.
\end{theorem}
\end{section}

\appendix

\begin{section}{Polynomials and finite dimensional
algebras}\label{Appendix A}

In this section, $F$ will denote an algebraically closed field of
characteristic zero.

Our goal here is to explain some of the basic ideas the relate the
structures of polynomials and finite dimensional algebras. Recall
that the Capelli polynomial $c_{n}$ is defined by
$$c_{n}=\sum_{\sigma\in \Sym(n)} \sgn(\sigma)x_{\sigma(1)}y_{1}x_{\sigma(2)}y_{2}\cdots
x_{\sigma(n)}y_{n}.$$

We say that the Capelli polynomial
alternates in the $x$'s. More generally, let
$f(X;Y)=f(x_{1},\ldots,x_{m}; Y)$ be a polynomial which is
multilinear in the set of variables $X$. We say that $f(X,Y)$ is
alternating in the set $X$ (or that the variables of $X$ alternate
in $f$) if there exists a polynomial
$h(X;Y)=h(x_{1},x_{2},\ldots,x_{m};Y)$ such that
$$f(x_{1},x_{2},\ldots,x_{m};Y)=\sum_{\sigma\in
\Sym(m)}\sgn(\sigma)h(x_{\sigma(1)},
x_{\sigma(2)},\ldots,x_{\sigma(m)};Y).$$

Following the notation in \cite{BR}, if $X_{1},
X_{2},\ldots,X_{p}$ are $p$ disjoint sets of variables,
we say that a polynomial
$f(X_{1},\ldots,X_{p};Y)$, is alternating in the sets $\{X_{1},\ldots,X_{p}\}$,
if it is alternating in each set $X_{i}$.

Clearly, in order to test whether a multilinear polynomial (and in
particular $c_{n}$) is an identity of a finite dimensional algebra
$A$, it is sufficient to evaluate the variables on basis elements.
It follows that if $A$ is an algebra over $F$ of dimension $n$
then $c_{n+1} \in \id(A)$. Clearly, we cannot expect that
Capelli's polynomial will detect precisely the dimension of a
finite dimensional algebra since on one hand we can just take a commutative
algebra $A$ of arbitrary dimension over $F$, and on the other hand $c_{2} \in
\id(A)$. This simple fact will lead us to consider (below) minimal
or adequate models.

Given an algebra $A$, finite dimensional over a field $F$, it is
well known that $A$ decomposes as a vector space into $A \cong
\overline{A} \oplus J(A)$ where $\overline{A}$ is semisimple and
$J(A)$ is the radical of $A$. Moreover, $\overline{A}$ is closed
under multiplication. As mentioned above, in order to test whether
a multilinear polynomial $f$ is an identity of $A$ it is
sufficient to evaluate the variables on any chosen basis of $A$
over $F$ and hence we may take a basis consisting of elements which belong either to
$\overline{A}$ or $J(A)$. We refer to these evaluations as
semisimple or radical evaluations respectively. Our aim is to
present a set of polynomials which detect the dimension of
$\overline{A}$ over $F$ and also the nilpotency index of $J(A)$.

Denote by $n=\dim_{F}(\overline{A})$ and by $s$ the nilpotency
index of $J(A)$.

For every integer $r$ consider the set of multilinear polynomials
with $r$-folds of alternating sets of variables of cardinality
$m$. Let us denote these sets of variables by $X_{1},
X_{2},\dots,X_{r}$. Clearly, if $n < m$ each alternating set will
assume at least one radical evaluation and hence if $r \geq s$ we
will have at least $s$ radical evaluations. This shows that the
polynomial vanishes upon any evaluation. It follows that if $n <
m$ and $r \geq s$ then $f$ is an identity of $A$. In other words
if we know that for every positive integer $r$ there is a
non-identity, multilinear polynomial with alternating sets of
cardinality $m$, then $n \geq m$. The question which arises
naturally is whether a finite dimensional algebra $A$, where
$n=\dim_{F}(\overline{A})$, will always admit, for arbitrary large
integer $r$, non-identities (multilinear) with $r$-alternating
sets of cardinality $n$. The answer is clearly negative since
again, on one hand we can take a semisimple commutative algebra of
arbitrary dimension over $F$, and on the other hand, the
cardinality of alternating sets cannot exceed $1$. Again, this
leads us to consider adequate models. The following terminology is
not standard and will be used only in this Appendix.

\begin{definition}

A finite dimensional algebra $A$ is \textit{weakly adequate} if
for every integer $r$ there is a multilinear polynomial,
non-identity of $A$, which has $r$-alternating sets of cardinality
$n=\dim_{F}(\overline{A})$.

\end{definition}

An important result due to Kemer implies:

\begin{lemma}(``Kemer's Lemma $1$")

Any finite dimensional algebra $A$ is \textit{PI}-equivalent (i.e. the
same $T$-ideal of identities) to a direct product $A_{1} \times
\cdots \times A_{k}$ where $A_{i}$ is weakly adequate.
\end{lemma}

The lemma allows us to control the dimension of the semisimple
component of $A$ (after passing to direct products of weakly
adequate algebras) in terms of noncommutative polynomials. But we
need more. We would like to control also the nilpotency index in
terms of noncommutative polynomials. For this we need to
strengthen the definition of weakly adequacy.

\begin{definition}

A finite dimensional algebra $A$ is \textit{adequate} if for every
integer $r$ there is a multilinear polynomial, non-identity of
$A$, which has $r$-alternating sets of cardinality
$n=\dim_{F}(\overline{A})$ and precisely $s-1$ alternating sets of
variables of cardinality $n+1$.

\end{definition}

As noted above a non identity of $A$ cannot have more than $s-1$
alternating sets of cardinality $n+1$.

A key result of Kemer (``Kemer's Lemma $2$") implies:

\begin{theorem}[Adequate model theorem]

Any finite dimensional algebra $A$ is \textit{PI}-equivalent to a direct
product $A_{1} \times \cdots \times A_{k}$ where $A_{i}$ is
adequate.
\end{theorem}

\begin{remark}
In fact one shows by a sequence of reductions, that any finite
dimensional algebra $A$ is \textit{PI}-equivalent to a direct product of
algebras which are called \textit{basic}. Kemer's Lemma 2 says that any
basic algebra is adequate.

\end{remark}

\begin{remark}

It should be emphasized that our main application of Kemer's Lemma
is in the ``reverse direction'': We start with $\Gamma$, the
$T$-ideal of identities of an affine algebra $W$. First one shows
that there exists a finite dimensional algebra $A$ such that
$\Gamma \supseteq \id(A)$. Then one shows easily that there exist
a pair $(n,s)$ of non-negative integers, such that for any integer
$r$ there exist polynomials $f$ \textit{outside} $\Gamma$ (called
Kemer polynomials for $\Gamma$) which have $r$ sets of alternating
variables of cardinality $n$ and $s-1$ sets of alternating
variables of cardinality $n+1$. Moreover the pair $(n,s)$ is
maximal with respect to the usual lexicographic order. The point
of Kemer's Lemma's is that one can find a basic algebra which
``realizes'' these parameters, i.e. a finite dimensional algebra
$A$ where $n=\dim_{F}(\overline{A})$, $s$ is the nilpotency index
of $J(A)$ and such that the Kemer polynomials for $\Gamma$ are
\textit{outside} $\id(A)$. As pointed out in the introduction,
this is the connection which allows us to prove the Phoenix
property for Kemer polynomials.

\end{remark}

\title{\textbf{Acknowledgment}: We thank Vesselin Drensky for
pointing out an error which occurred in Theorem
\ref{PI-equivalence-general} in an earlier version of this paper.}

\end{section}

\end{document}